\documentclass{amsproc}

\usepackage{amsfonts}
\usepackage{amsmath}
\usepackage{amssymb}
\usepackage{amsthm}
\usepackage{latexsym}
\usepackage{longtable}
\usepackage{graphicx}
\usepackage{caption}
\usepackage{subfigure}
\usepackage{xfrac}
\usepackage{fix-cm}
\usepackage{hyperref,doi}
\usepackage{url}
\usepackage{stmaryrd}
\usepackage{multirow}
\usepackage{tabularx}
\usepackage{tikz-cd}
\usepackage{array}
\usepackage[inline]{enumitem}
\usepackage{color}
\usepackage{textcomp}

\newtheorem{theorem}{Theorem}[section]
\newtheorem{corollary}[theorem]{Corollary}

\newtheorem{lemma}[theorem]{Lemma}
\newtheorem{example}[theorem]{Example}
\newtheorem{definition}[theorem]{Definition}

\newtheorem{conventions}[theorem]{Conventions}
\newtheorem{remark}[theorem]{Remark}

\numberwithin{equation}{section}
\numberwithin{figure}{section}

\newcommand{\pa}[1]{\left(#1\right)}

\newcommand{\tn}[1]{\textnormal{#1}}

\newcommand{\card}[1]{\left| #1 \right|}

\newcommand{\CAT}{\textsc{cat}}
\newcommand{\TOP}{\textsc{top}}
\newcommand{\DIFF}{\textsc{diff}}
\newcommand{\PL}{\textsc{pl}}
\newcommand{\R}{\mathbb{R}}
\newcommand{\intr}[1]{#1^{\circ}}
\newcommand{\closr}[1]{\overline{#1}}
\newcommand{\bd}{\partial}
\newcommand{\E}[1]{\mathcal{E}\left(#1\right)}
\newcommand{\enbd}[2]{\mathcal{N}(#1;#2)}
\newcommand{\eh}{h}
\renewcommand{\phi}{\varphi}
\newcommand{\defword}{\textbf}
\renewcommand{\setminus}{-}

\newcommand{\piz}[1]{\card{\pi_0\pa{#1}}}
\newcommand{\parity}[1]{\mathcal{P}(#1)}
\newcommand{\image}[1]{\textnormal{Im}\,#1}
\newcommand{\bdnc}[1]{\mathrm{d}#1}

\begin{document}

\title{The End Sum of Surfaces}

\author[L.~Axon]{Liam Axon}
\address{Newton, MA 02459}
\email{laxon26@gmail.com}

\author[J.~Calcut]{Jack Calcut}
\address{Department of Mathematics, Oberlin College, Oberlin, OH 44074}
\email{jcalcut@oberlin.edu}



\date{\today}

\begin{abstract}
End sum is a natural operation for combining two noncompact manifolds and
has been used to construct various manifolds with interesting properties.
The uniqueness of end sum has been well-studied in dimensions three and higher.
We study end sum---and the more general notion of adding a 1-handle at infinity---for surfaces
and prove uniqueness results.
The result of adding a 1-handle at infinity to distinct ends of a surface with compact boundary
is uniquely determined by the chosen ends and the orientability of the 1-handle.
As a corollary, the end sum of two surfaces with compact boundary is uniquely determined by the chosen ends.
Unlike uniqueness results in higher dimensions, which rely on isotopy uniqueness of rays,
our results rely fundamentally on a classification of noncompact surfaces.
\end{abstract}

\maketitle

\section{Introduction} \label{sec:introduction}

End sum is the analogue for open manifolds of the boundary sum of manifolds with boundary.
It was introduced by Gompf~\cite{gompf3,gompfinf} in the 1980's to construct smooth manifolds homeomorphic to $\R^4$.
Gompf~\cite[p.~322]{gompf3} colloquially described end sum as
gluing together two noncompact manifolds using a piece of tape.
More formally, the piece of tape is a $1$-handle at infinity.
Since that time, several authors have used end sum to construct manifolds with interesting properties.
Recently, Bennett~\cite{bennett} used end sum to produce new smooth structures on open $4$-manifolds
and Sparks~\cite{sparks} used it to construct $4$-dimensional splitters.
For further examples, see the second author and Gompf~\cite{cg}
and the second author, Haggerty, and Guilbault~\cite{cgh}.

A 1-handle at infinity is attached to a manifold $M$ along a chosen pair of disjoint, properly embedded rays
pointing to ends of $M$.
End sum is the special case where $M$ has two components and the 1-handle connects them.
The dependence on ray choice of adding a 1-handle at infinity has been well-studied.
Gompf~\cite{gompfinf} first showed that end sums of manifolds homeomorphic to $\R^4$
are independent of ray choice.
Myers~\cite{myers} showed that end summing two copies of $\R^3$ using knotted rays
yields uncountably many homeomorphism types of contractible, open 3-manifolds. 
The second author and Haggerty~\cite{ch} constructed examples of pairs of
connected, open, oriented, one-ended $n$-manifolds for each $n\geq3$
that may be end summed using various rays to produce manifolds that are not proper homotopy equivalent.
The latter examples arise from complicated fundamental group behavior at even just one of the ends.
In dimension $n=1$, adding a $1$-handle along specified ends is always unique.

For manifolds of any dimension, the \textit{Mittag-Leffler} condition (also called \textit{semistability})
on an end is a necessary and sufficient condition for any two proper rays pointing to that end to be properly homotopic.
In dimensions four and higher, such a proper homotopy may be upgraded to an ambient isotopy.
That yields uniqueness results for end sums and adding 1-handles at infinity as proved by the second author and Gompf~\cite{cg}.

\begin{theorem}[Calcut and Gompf]\label{thm:cg}
	Let $X$ be a (possibly disconnected) $n$-manifold, $n\ge4$.
	Then the result of attaching a (possibly infinite) collection of 1-handles at infinity to some oriented
	Mittag-Leffler ends of $X$ depends only on the pairs of ends to which each 1-handle is attached,
	and whether the corresponding orientations agree.
\end{theorem}

An immediate corollary is the uniqueness of oriented end sums of two $n$-manifolds with
$n\geq4$ along oriented Mittag-Leffler ends.

In the present paper, we study end sum for surfaces.
Note from the outset that for each end of a \textit{surface with compact boundary},
the Mittag-Leffler condition is actually equivalent to the end being collared.
An end of a manifold is \textit{collared} provided it has a neighborhood
of the form $Z\times [0,\infty)$ for some connected, compact manifold $Z$.
The authors plan to address the classification of Mittag-Leffler ends of general surfaces in a future paper.
Observe now that a Mittag-Leffler end of a surface with noncompact boundary components
need not be collared as shown by the end $\eta$ in Figure~\ref{fig:oddex1}.
\begin{figure}[htb!]
    \centerline{\includegraphics[scale=0.9]{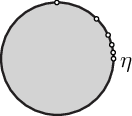}}
    \caption{Surface $M$ obtained from the closed disk by removing a sequence of boundary points and the single limit point of that sequence.}
	\label{fig:oddex1}
\end{figure}
Properly embedded rays in $\R^2$, $S^1\times[0,\infty)$, and closed half space $\R^2_+$
may be straightened by ambient isotopy~\cite[pp.~1845--1852]{cks}.
So, up to isotopy, there are unique proper rays in Mittag-Leffler (=collared) ends of surfaces with compact boundary.
In general, ends of noncompact surfaces may have infinite genus, complicated fundamental group behavior,
and fail to be Mittag-Leffler.
Thus, one might suspect that end sums of surfaces depend on the choices of rays within specified ends.
In fact, the contrary is true. The following is our main result.

\begin{theorem} \label{mainthm}
	Let $M$ be a (possibly disconnected) surface with compact boundary.
	Then the result of attaching a 1-handle at infinity to distinct ends of $M$
	depends only on the ends to which the 1-handle is attached and orientations.
	If the chosen ends lie in different components of $M$ or they lie in the same non-orientable component of $M$,
	then orientation is irrelevant.
\end{theorem}

A more precise statement is given in Theorem~\ref{thm:endSumUniquenessPLCase} below.
An immediate corollary is uniqueness of end sums of two connected surfaces with compact boundary along chosen ends
(irrespective of orientations).

Contrasting Theorems~\ref{thm:cg} and~\ref{mainthm}, we see that for surfaces the relevant ends need not be Mittag-Leffler
and there is greater flexibility with orientations.
On the other hand, our arguments for surfaces assume: (i) $M$ has compact boundary,
(ii) a single $1$-handle is attached to $M$,
and (iii) the relevant ends of $M$ are distinct.
Throughout, we opt to work with a single $1$-handle for simplicity.
Applying our results iteratively yields results for attaching finitely many $1$-handles,
and our results likely carry over to appropriate settings involving infinitely many $1$-handles.
Before we discuss assumptions (i) and (iii), we introduce some useful terminology.

Consider the general case of attaching a single $1$-handle at infinity to an $n$-manifold $M$.
Let $r$ and $r'$ denote the rays in $M$ along which the $1$-handle is attached,
let $\varepsilon$ and $\varepsilon'$ denote the ends of $M$ to which $r$ and $r'$ point (respectively),
and let $N=M\cup \pa{1\tn{-handle}}$ denote the resulting $n$-manifold.
In this unrestricted setting, we allow $M$ to be disconnected, $\partial M$ to be compact (possibly empty) or
noncompact (possibly with noncompact boundary components),
$\varepsilon$ and $\varepsilon'$ to be equal or distinct,
and the $1$-handle attachment to respect or ignore any possible given orientations.
An end $\eta$ of $N$ is \textit{ordinary} provided it has a neighborhood disjoint from the $1$-handle.
Otherwise, $\eta$ is \textit{extraordinary}.
Intuitively, the extraordinary end(s) of $N$ are those involved in the attachment of the $1$-handle.
For a space $X$, let $\E{X}$ denote the \textit{space of ends} of $X$,
and let $\card{\E{X}}$ denote the \textit{number of ends} of $X$.
The possible number of extraordinary ends of $N$ varies with the dimension $n$.
If $n=1$, then there are no extraordinary ends---attaching a $1$-handle at infinity
to a $1$-manifold simply eliminates two ends of $M$ and $\card{\E{N}}=\card{\E{M}}-2$.
If $n\geq3$, then there is one extraordinary end, and
$\card{\E{N}}=\card{\E{M}}-1$ (when $\varepsilon\ne\varepsilon'$) or
$\card{\E{N}}=\card{\E{M}}$ (when $\varepsilon=\varepsilon'$).
For surfaces, the situation is more complicated---ultimately due to the fact that
a surface may be separated by a $1$-dimensional submanifold.
There may be one or two extraordinary ends, and $\card{\E{N}}$
may equal $\card{\E{M}}-1$, $\card{\E{M}}$, or $\card{\E{M}}+1$.
Predicting which occur is subtle, especially in the presence of noncompact boundary components.
The following examples exhibit that subtlety and others.
Basic examples are included for context and comparison.
In each example, $n=2$ and $N=M\cup \pa{1\tn{-handle}}$.

\begin{enumerate}
\item Let $M=\R^2 \sqcup \R^2$ be the disjoint union of two copies of $\R^2$ as in Figure~\ref{fig:r2basic}.
\begin{figure}[htb!]
    \centerline{\includegraphics[scale=0.9]{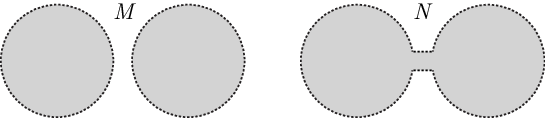}}
    \caption{Surface $M=\R^2 \sqcup \R^2$ (left) and end sum $N$ of $M$ (right).}
	\label{fig:r2basic}
\end{figure}
Regardless of the orientability of the $1$-handle, $N$ is homeomorphic to $\R^2$ 
and has one end (extraordinary).
Here, $\card{\E{N}}=\card{\E{M}}-1$.
\item Let $M=\R^2$ as in Figure~\ref{fig:r2exs}.
\begin{figure}[htb!]
    \centerline{\includegraphics[scale=0.9]{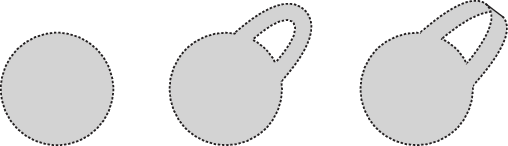}}
    \caption{Surface $M=\R^2$ (left), result $N$ of attaching an oriented $1$-handle (middle),
				and result $N'$ of attaching a non-oriented $1$-handle (right).}
	\label{fig:r2exs}
\end{figure}
Let $N$ be the result using an oriented $1$-handle,
and let $N'$ be the result using a non-oriented $1$-handle.
Then, $N$ is an open cylinder with two ends (both extraordinary),
and $N'$ is an open M\"{o}bius band with one end (extraordinary).
Here, $\card{\E{N}}=\card{\E{M}}+1$ and $\card{\E{N'}}=\card{\E{M}}$.
\item Let $M$ be an open annulus as in Figure~\ref{fig:annulusbasic}.
\begin{figure}[htb!]
    \centerline{\includegraphics[scale=0.9]{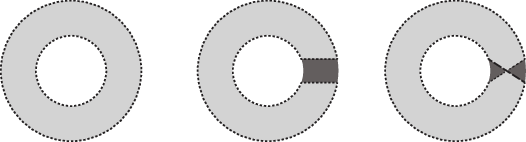}}
    \caption{Open annulus $M$ (left),
				result $N$ of attaching an oriented $1$-handle (middle),
				and result $N'$ of attaching a non-oriented $1$-handle (right).}
	\label{fig:annulusbasic}
\end{figure}
Attach a $1$-handle to the distinct ends of $M$.
Let $N$ be the result using an oriented $1$-handle,
and let $N'$ be the result using a non-oriented $1$-handle.
Then, $N$ is a punctured torus with one end (extraordinary),
and $N'$ is a punctured Klein bottle with one end (extraordinary).
Here, $\card{\E{N}}=\card{\E{M}}-1=\card{\E{N'}}$.
\item Let $M=\R^2_+ \sqcup \R^2_+$ be the disjoint union of two copies of closed half-space
as in Figure~\ref{fig:endSumBoundaryCounterexample}.
\begin{figure}[htb!]
    \centerline{\includegraphics[scale=0.9]{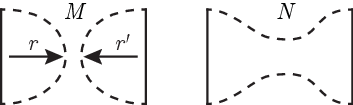}}
    \caption{Surface $M=\R^2_+ \sqcup \R^2_+$ and end sum $N$ of $M$.}
\label{fig:endSumBoundaryCounterexample}
\end{figure}
Regardless of the orientability of the $1$-handle, $N$ is homeomorphic to
$[0,1]\times\R$ and has two ends (both extraordinary).
Here, $\card{\E{N}}=\card{\E{M}}$ in both cases.
In a sense, the noncompact boundary components of $M$ clog up the ends of $N$.
\item\label{deex} Let $M$ be the open, oriented, one-ended surface with infinite genus as in Figure~\ref{fig:sameendex}.
\begin{figure}[htb!]
    \centerline{\includegraphics[scale=0.9]{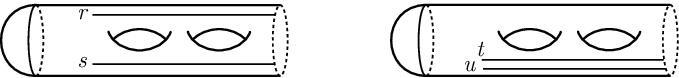}}
    \caption{One-ended, infinite genus surface $M$ containing two non-parallel rays $r$ and $s$ (left)
						and two parallel rays $t$ and $u$ (right).}
	\label{fig:sameendex}
\end{figure}
Phillips and Sullivan~\cite{ps} referred to $M$ as the \textit{Infinite Loch Ness monster};
see also Aramayona and Vlamis~\cite[p.~463]{av}.
Attach an oriented $1$-handle at infinity to $M$.
Let $N$ be the result using the non-parallel rays $r$ and $s$, and
let $N'$ be the result using the parallel rays $t$ and $u$.
Then $N$ has one end (infinite genus and extraordinary),
and $N'$ has two ends (one of genus zero, the other of infinite genus, and both extraordinary).
Here, $\card{\E{N}}=\card{\E{M}}$ and $\card{\E{N'}}=\card{\E{M}}+1$.
Thus, the hypothesis in Theorem~\ref{mainthm} that the $1$-handle is attached to distinct ends may not be omitted.
In fact, the distinct end hypothesis is required even when $M$ has finite genus.
For a planar example, let $M$ be $\R^2$ with the integer points on the $x$-axis removed.
So, $M$ has infinitely many ends, exactly one of which is not isolated.
Let $r$ and $s$ be non-parallel rays in the positive and negative $y$-axes respectively.
Let $t$ and $u$ be parallel rays in the upper half-plane.
Attach an oriented $1$-handle at infinity to $M$.
Let $N$ be the result using the non-parallel rays $r$ and $s$, and
let $N'$ be the result using the parallel rays $t$ and $u$.
Then $N$ and $N'$ are not homeomorphic since $N'$ has one nonisolated end whereas $N$ has two.
\item Let $X$ be the $2$-disk $D^2$ with three points removed from its boundary.
So, $X$ has three noncompact boundary components and three ends.
Each end of $X$ has genus zero and an open neighborhood homeomorphic to $[0,1]\times\R$.
Let $Y$ be obtained from $X$ by connect summing a sequence of tori converging to one end.
The surface $Y$ is depicted on the left in Figure~\ref{fig:endsumexncb}
where the zeros indicate genus zero ends, $\infty$ indicates the infinite genus end,
and the segment indicates a ray.
\begin{figure}[htb!]
	\centering
	\subfigure[Surface $M=Y\sqcup Y$ to be end summed along indicated rays.]
	{
	    \label{fig:endsumexncb1}
	    \includegraphics[scale=0.9]{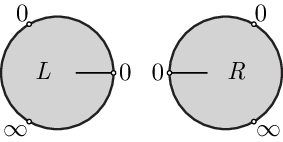}
	}
	\hspace{0.25cm}
	\subfigure[End sum $N$ using an oriented $1$-handle.]
	{
	    \label{fig:endsumexncb2}
	    \includegraphics[scale=0.9]{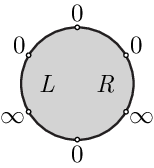}
	}
	\hspace{0.25cm}
	\subfigure[End sum $N'$ using a non-oriented $1$-handle.]
	{
	    \label{fig:endsumexncb3}
	    \includegraphics[scale=0.9]{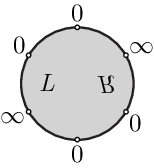}
	}
	\caption{Orientation of the $1$-handle at infinity is relevant for end sums of surfaces with noncompact boundary components.}
	\label{fig:endsumexncb}
\end{figure}
The surface $M=Y\sqcup Y$ is the disjoint union of two copies of $Y$.
End sum $M$ along the indicated rays.
Let $N$ be the result using an oriented $1$-handle, and
let $N'$ be the result using a non-oriented $1$-handle.
The surfaces $N$ and $N'$ are shown in Figure~\ref{fig:endsumexncb},
where the letters $L$ and $R$ are included to display orientations.
Both $N$ and $N'$ have six ends (four of genus zero, two of infinite genus, and two extraordinary both of genus zero).
Here, $\card{\E{N}}=\card{\E{M}}=\card{\E{N'}}$.
Notice that $N$ and $N'$ are not homeomorphic---$N$ has two noncompact boundary components
that point only to genus zero ends of $M$ and point to a common end of $M$,
whereas $N'$ does not have such boundary components.
Up to homeomorphism, there are exactly two end sums, namely $N$ and $N'$, of $M$ along the specific ends just used
since ray choice is irrelevant at those ends.
To obtain similar examples with $M$ non-orientable,
remove an open disk from the interior of $Y$ and glue in a crosscap.
Define $N$ and $N'$ as before.
Though both components of $M$ are non-orientable, and often orientability of the $1$-handle is not relevant
when $M$ is non-orientable, the end sums $N$ and $N'$ remain non-homeomorphic (the same argument still applies).
Thus, for end sums of surfaces with noncompact boundary components, orientability of the $1$-handle is relevant.
\end{enumerate}

Example~\ref{deex} showed that Theorem~\ref{mainthm} is false without the distinct ends hypothesis.
To remove that hypothesis, appropriate replacement hypotheses would be necessary.
Our proof of Theorem~\ref{mainthm} proceeds by studying end invariants of the extraordinary end of $N$
and then applying the classification of noncompact surfaces with compact boundary as proved by Richards~\cite{richards}.
One may attempt to remove from Theorem~\ref{mainthm} the hypothesis that $\partial M$ is compact by instead
using the classification of noncompact surfaces with possibly noncompact boundary
due to Brown and Messer~\cite{brownmesser}.
Even with that approach to the case where $\partial M$ is noncompact,
one must understand how ray choice affects the number of extraordinary ends
and the number of ends of $N$, and how the orientation of the $1$-handle affects $N$.
The examples above indicate that those may be subtle questions.

To circumvent those nuances, the present paper focuses on the case where a single $1$-handle at infinity is attached to
a surface $M$ with compact boundary along rays pointing to distinct ends of $M$.
In that case, there is a unique extraordinary end of $N=M\cup \pa{1\tn{-handle}}$ and $\card{\E{N}}=\card{\E{M}}-1$.
We allow $M$ to be disconnected and $\partial M$ to be empty.
Without loss of generality, it suffices to consider two cases: (i) $M$ is connected,
and (ii) $M$ has two connected components and the $1$-handle connects them.
The latter operation is the end sum of the two components of $M$.

This paper is organized as follows.
Section~\ref{sec:endsIntroduction} defines an end and recalls the theory of ends sufficient for our purposes.
We have included that material to help make this paper more self-contained and accessible.
Section~\ref{sec:endSumIntroduction} defines end sum and $1$-handle addition at infinity,
sets up notation for those operations, and proves that those operations indeed yield manifolds
(we are unaware of a published proof of this fact).
Section~\ref{sec:classification} recalls the classification of noncompact surfaces with compact boundary,
including generalized genus, parity, end invariants, and orientability.
Section~\ref{sec:mainTheorem} proves our main result---Theorem~\ref{thm:endSumUniquenessPLCase}---for \PL\
surfaces by studying how each of the following end invariants are affected by the addition of a $1$-handle at infinity:
the space of ends, boundary, orientability, genus, and parity.
Several results in Section~\ref{sec:mainTheorem} are proved more generally for $n$-manifolds.
In particular, Lemma~\ref{lem:partEndSumEnds} shows that the space of ends of $N=M\cup \pa{1\tn{-handle}}$
is the quotient space of the space of ends of $M$ by identifying the ends of $M$ along which the $1$-handle is attached
(see Lemma~\ref{lem:partEndSumEnds} for a more precise statement).
Section~\ref{sec:rayUniquenessMain} provides an alternative proof of Theorem~\ref{thm:endSumUniquenessPLCase}
using Brown and Messer's~\cite{brownmesser} classification of noncompact surfaces.
That approach also yields a ray uniqueness result for surfaces---see Theorem~\ref{thm:rayUniqueness}.
Lastly, Section~\ref{sec:secondTheorem} explains how to extend Theorem~\ref{thm:endSumUniquenessPLCase}
to \TOP\ and \DIFF\ surfaces.

We use the following conventions where
$X$ is a topological space, $A$ is any subspace of $X$ (denoted $A\subset X$), and $M$ is a manifold.
\begin{itemize}
    \item $\closr{A}$ denotes the topological \defword{closure} of $A$ in $X$.
    \item $\intr{A}$ denotes the topological \defword{interior} of $A$ in $X$.
    \item $\R_+$ denotes $[0,\infty)$ and $\mathbb{R}^n_+$ denotes $[0, \infty) \times \mathbb{R}^{n-1}$.
    \item $D^n$ denotes the closed $n$-disk.
    \item \CAT\ denotes one of the manifold categories: \TOP, \PL, or \DIFF.
    \item \CAT\ manifolds are Hausdorff and paracompact, possibly with boundary.
		\item A connected manifold without boundary is \defword{closed} provided it is compact and is \defword{open} provided it is noncompact.
    \item $M \cong M'$ denotes \defword{isomorphism} of \CAT\ manifolds.
    \item $\bd{M}$ denotes the \defword{manifold boundary} of $M$.
    \item $\E{X}$ denotes the \defword{space of ends} of $X$.
    \item If $\alpha$ is an end of $X$ and $K$ is a compact subspace of $X$,
					then $\enbd{\alpha}{K}$ denotes the set of all ends $\beta$ of $X$ for which $\beta(K) = \alpha(K)$.
    \item $b(M) = \piz{\bd{M}}$ denotes the \defword{number of boundary components} of $M$.
    \item $g(M)$ denotes the \defword{genus} of a surface $M$ and $\parity{M}$ denotes the parity of $M$.
\end{itemize}

\noindent\textbf{Acknowledgments} The authors are grateful to Craig Guilbault for suggesting the main problem studied herein
and to Ric Ancel for suggesting a fruitful alternative approach to prove the main theorem based on uniqueness of rays in a surface.
The authors also thank an anonymous referee for several useful comments.

\section{Ends of Spaces}\label{sec:endsIntroduction}

Loosely speaking, an end of a space may be thought of as an infinity of the space.
For example, the closed interval has no ends, a ray has one end, and the real line has two ends.
Figure~\ref{fig:endExamples} shows three manifolds and their ends.
\begin{figure}[htb!]
	\centering
	\subfigure[The real line has two ends.]
	{
	    \label{fig:rends}
	    \includegraphics[scale=0.75]{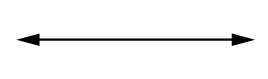}
	}
	\hspace{0.25cm}
	\subfigure[The thrice punctured sphere has three ends.]
	{
	    \label{fig:thriceends}
	    \includegraphics[scale=0.75]{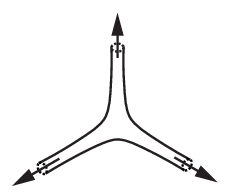}
	}
	\hspace{0.25cm}
	\subfigure[The closed disk has no ends.]
	{
	    \label{fig:diskends}
	    \includegraphics[scale=0.9]{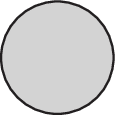}
	}
	\caption{Manifolds with their ends indicated by arrows.}
	\label{fig:endExamples}
\end{figure}
Each (suitably nice) space is compact if and only if it has no ends.
The set of all ends of a space may be equipped with a natural topology.
Ends and the space of ends play fundamental roles in the study of noncompact spaces.
A prime example is Whitehead's open, contractible 3-manifold,
the first example of an open, contractible manifold not isomorphic to Euclidean space.
It may be distinguished from $\R^3$ by properties of its end (see Guilbault~\cite[pp.~6--7]{guilbault}).
The space of ends is also an essential ingredient in Richards'~\cite{richards}
classification of noncompact surfaces which we review in Section~\ref{sec:classification}.

The theory of ends dates back to Freudenthal and Hopf in the 1930's.
For interesting further reading, see Freudenthal's original paper~\cite{freudenthal},
Siebenmann's thesis~\cite[pp.~8--12]{siebenmannthesis},
and Guilbault's chapter~\cite{guilbault}.
As we are focused on manifolds, we restrict our study to spaces that satisfy the following condition.

\begin{definition}
	A topological space $X$ is \defword{nice for ends} provided $X$ is Hausdorff, locally compact, $\sigma$-compact, connected, and locally connected.
\end{definition}

We adopt the convention that a space $X$ is \defword{connected} provided
it has exactly two subspaces that are both open and closed in $X$,
namely $X$ and $\varnothing$.
In particular, the empty space is neither connected nor disconnected.

Connected manifolds are nice for ends, as are connected, locally finite simplicial complexes.
Each space that is nice for ends is necessarily paracompact.
However, it is not necessarily separable, metrizable, or even first-countable.
For example, the product of uncountably many copies of $[0,1]$ is compact and nice for ends, but does not satisfy those three properties.

Throughout this section---unless explicitly stated otherwise---$X$ is a space nice for ends.

\begin{definition}\label{defend}
	An \defword{end} of $X$ is any function $\epsilon$ defined on the collection of compact subspaces of $X$ such that:
	for each compact $K \subset X$, the output $\epsilon(K)$ is a connected component (hence nonempty) of $X \setminus K$,
	and if $K_1 \subset K_2$, then  $\epsilon(K_2) \subset \epsilon(K_1)$.
	Let $\E{X}$ denote the set of all ends of $X$.
\end{definition}

For example, $\R$ has two ends $\epsilon_-$ and $\epsilon_+$.
Given a compact $K \subset \R$, $\epsilon_-(K)$ is the connected component of
$\R \setminus K$ that contains elements less than $K$,
and $\epsilon_+(K)$ is the connected component of $\R \setminus K$ that contains elements greater than $K$.
One may verify that $\epsilon_-$ and $\epsilon_+$ are ends of $\R$ and, in fact, are the only ends of $\R$.

We recall some fundamental properties of ends. Several proofs will be left to the reader.

\begin{lemma} \label{lem:intersectionOfEndNeighborhoods}
	If $\epsilon$ is an end of $X$, and $K$ and $K'$ are compact subspaces of $X$, then $\epsilon(K) \cap \epsilon(K') \neq \varnothing$.
\end{lemma}

\begin{lemma} \label{lem:compactSpacesHaveNoEnds}
	If $X$ is compact, then $X$ has no ends.
\end{lemma}

The converse of Lemma~\ref{lem:compactSpacesHaveNoEnds} also holds---see Corollary~\ref{cor:noncompactImpliesHasAnEnd} below.

\begin{definition}
	A subspace of $X$ is \defword{bounded} provided its closure in $X$ is compact, and \defword{unbounded} otherwise.
\end{definition}

This terminology aligns with the definition of a bounded subset of Euclidean space,
although in general metric spaces the notions can differ.

\begin{lemma} \label{lem:endsAreUnbounded}
	Let $\epsilon$ be an end of $X$. If $K\subset X$ is compact, then $\epsilon(K)$ is an unbounded subspace of $X$.
\end{lemma}

\begin{lemma} \label{lem:unboundedComponents}
	Let $K \subset X$ be compact. Then, $X \setminus K$ has finitely many unbounded connected components.
	Furthermore, if $A$ is the union of all bounded connected components of $X \setminus K$,
	then $K \cup A$ is compact and $X \setminus (K \cup A)$ has only unbounded connected components.
\end{lemma}

For the proof of Lemma~\ref{lem:unboundedComponents}, we offer this hint:
use local compactness to construct a bounded open neighborhood $V$ of $K$.
Then, look at the open cover of $\overline{V}$ consisting of $V$ and all components of $X \setminus K$.

\begin{example}
Let $X=D^2$. For each integer $n \geq 1$, let $p_n = (0, 2^{-n})$ and let $U_n$ be an open disk of radius $4^{-n}$ centered at $p_n$.
The space $K = X \setminus \cup_n U_n$ (a nonsurface) is compact, but each $U_n$ is a bounded connected component of $X \setminus K$.
Thus, although there are always finitely many unbounded connected components of $X \setminus K$,
there may be infinitely many bounded components.
\end{example}

The notion of a \textit{neighborhood of an end} is used to define a topology on the set of ends of $X$.

\begin{definition}
	Let $\epsilon$ be an end of $X$.
	A \defword{neighborhood} of $\epsilon$ is a subspace $A \subset X$ such that
	there exists some compact subspace $K \subset X$ for which $\epsilon(K) \subset A$.
\end{definition}

For example, consider $X = \R$ with ends $\epsilon_-$ and $\epsilon_+$.
A subspace of $\R$ is a neighborhood of $\epsilon_-$ if and only if it contains $(-\infty, a)$ for some $a \in \R$.
Similarly, a subspace of $\R$ is a neighborhood of $\epsilon_+$ if and only if it contains $(b, \infty)$ for some $b \in \R$.

\begin{lemma}
	If $K \subset X$ is compact, then $X \setminus K$ is a neighborhood of every end of $X$.
\end{lemma}

\begin{lemma}
	If $\epsilon_1$ and $\epsilon_2$ are distinct ends of $X$,
	then there exist disjoint neighborhoods $A_1$ of $\epsilon_1$ and $A_2$ of $\epsilon_2$ (so, $A_1 \cap A_2 = \varnothing$).
\end{lemma}

An alternative definition of an end may given using a \textit{compact exhaustion}.
This important and equivalent definition is useful for constructing and visualizing ends.

\begin{definition}
	A \defword{compact exhaustion} of a topological space $X$ is
	a sequence $(K_i)$ of compact subspaces $K_i\subset X$ such that $K_i \subset K_{i+1}$ for all $i$ and
	\[\bigcup_{i=1}^\infty \intr{K_i} = X\]
\end{definition}

For example, $K_i = [-i, i]$ is a compact exhaustion of $\R$.
See Figure~\ref{fig:exhaustionexamples} for two more examples.
\begin{figure}[htb!]
	\centering
	\subfigure[Thrice punctured sphere.]
	{
	    \label{fig:thriceexhaust}
	    \includegraphics[scale=0.75]{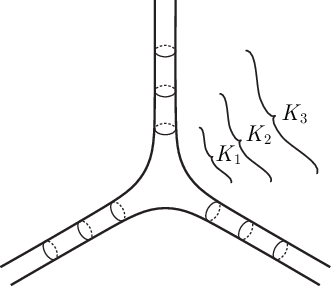}
	}
	\hspace{0.25cm}
	\subfigure[Infinite tree.]
	{
	    \label{fig:inftreeexhausr}
	    \includegraphics[scale=0.8]{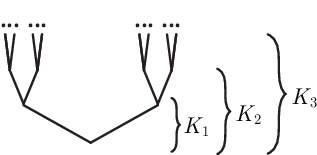}
	}
	\caption{Compact exhaustions.}
	\label{fig:exhaustionexamples}
\end{figure}
To make use of compact exhaustions, we need their existence.

\begin{lemma}
The (nice for ends) space $X$ has a compact exhaustion.
\end{lemma}

\begin{proof}
As $X$ is $\sigma$-compact, there exist compact subspaces $L_1$, $L_2$, $\ldots$ of $X$ such that $\cup_j L_j = X$.
As $X$ is locally compact, each $L_j$ has a compact neighborhood $N_i$ in $X$ (meaning $L_j\subset \intr{N_i}$).
Define $K_i = \cup_{j \leq i} N_j$, which is compact.
For each $x \in X$, there exists $j$ such that $x \in L_j$.
So, $x \in \intr{K_i}$ and $(K_i)$ is a compact exhaustion of $X$.
\end{proof}

A compact exhaustion $(K_i)$ of $X$ describes $X$ as a limit of the compact sets $K_i$.
It may also be used to define an end of $X$ as a limit of unbounded components of the complements $X \setminus K_i$.

Fix a compact exhaustion $(K_i)$ of $X$.
Define $\mathcal{E}'(X)$ to be the set of all sequences $(V_i)$ such that $V_i$ is a connected (hence nonempty)
component of $X \setminus K_i$, and $V_{i+1} \subset V_i$.

\begin{lemma}
	There is a one-to-one correspondence between $\E{X}$ and $\mathcal{E}'(X)$
	that associates to each end $\epsilon$ the sequence $V_i = \epsilon(K_i)$.
\end{lemma}

\begin{proof}
If $\epsilon \in \E{X}$, then the sequence $V_i = \epsilon(K_i)$ is in fact in $\mathcal{E}'(X)$.
On the other hand, given any sequence $(V_i) \in \mathcal{E}'(X)$, one may associate an end $\epsilon$ as follows.
For each compact $L\subset X$, choose $i$ such that $L \subset K_i$.
Then, $V_i$ is a connected component of $X \setminus K_i$ and is contained in a unique connected component $W$ of $X \setminus L$.
Define $\epsilon(L) = W$. One may verify that $\epsilon$ is well-defined, and that these associations are inverses.
\end{proof}

When a compact exhaustion $(K_i)$ is given for $X$,
we will often conflate an end of $X$ with the sequence $(V_i) \in \mathcal{E}'(X)$.
Some authors---for example Guilbault~\cite[$\S$~3.3]{guilbault}---take this sequence as the definition of an end of $X$.
For alternative definitions, see Geoghegan~\cite[$\S$~13.4]{geoghegan}
and Aramayona and Vlamis~\cite{av}.
As an example, consider $X = \R$ with the compact exhaustion $K_i = [-i, i]$.
The only sequences in $\mathcal{E}'(X)$ are $V_i = (-\infty, i)$ and $V_i = (i, \infty)$.
These represent the two ends of $\R$.
Using compact exhaustions, we give a straightforward proof that each space with no ends is compact.

\begin{lemma} \label{lem:existenceOfEnds}
	Fix a compact exhaustion $(K_i)$ of $X$.
	If $V_1 \supset V_2 \supset \dots \supset V_n$ is a finite sequence
	where each $V_i$ is an unbounded connected component of $X \setminus K_i$,
	then this sequence extends to an infinite sequence $(V_i)$ which represents an end of $X$.
\end{lemma}

\begin{proof}
It suffices to prove that if $V_i$ is an unbounded connected component of $X \setminus K_i$,
then there exists an unbounded connected component of $X \setminus K_{i+1}$ contained in $V_i$.
Using this, any finite initial sequence may be extended inductively to an infinite sequence.

Let $L$ be the union of $K_{i+1}$ and all of the bounded components of $X \setminus K_{i+1}$.
By Lemma~\ref{lem:unboundedComponents}, $L$ is compact.
As $V_i$ is unbounded, $V_i \setminus L \neq \varnothing$.
Let $y \in V_i \setminus L$.
Then $y$ is contained in some component $V_{i+1}$ of $X \setminus K_{i+1}$.
As $y \not\in L$, this component must be unbounded.
\end{proof}

\begin{corollary} \label{cor:existenceOfEnds2}
	If $K \subset X$ is compact and $V$ is an unbounded connected component of $X \setminus K$,
	then there exists an end $\epsilon \in \E{X}$ such that $\epsilon(K) = V$.
\end{corollary}

\begin{proof}
Let $(K_i)$ be any compact exhaustion of $X$ such that $K_1 = \varnothing$, and let $K_i' = K_i \cup K$ for each $i$.
Note that $(K_i')$ is a compact exhaustion of $X$ and $K_1' = K$.
Apply Lemma~\ref{lem:existenceOfEnds} to the single-term initial sequence $(V)$.
\end{proof}

\begin{corollary} \label{cor:noncompactImpliesHasAnEnd}
	If $X$ is noncompact, then $X$ has at least one end.
\end{corollary}

\begin{proof}
Apply Corollary~\ref{cor:existenceOfEnds2} to $K = \varnothing$ and $V = X$.
\end{proof}

Compact exhaustions also provide an upper bound on the number of ends of a space.

\begin{lemma} \label{lem:upperBoundOnEnds}
	The number of ends of $X$ is at most the continuum.
\end{lemma}

\begin{proof}
Fix a compact exhaustion $(K_i)$ of $X$.
There exists one end for each choice of sequence $(V_i)$ where $V_i$ is a connected component of $X \setminus K_i$ and $V_i \supset V_{i+1}$. By Lemma~\ref{lem:unboundedComponents}, there exist only finitely many unbounded connected components of $X \setminus K_i$. There are at most continuum many ways to pick one element from each of a countable number of finite sets.
\end{proof}

The set of ends of $X$ admits a natural topology.
The resulting \textit{space of ends} is essential for the classification of noncompact surfaces.

\begin{definition}
	Let $\epsilon$ be an end of $X$, and let $K \subset X$ be compact.
	Let $\enbd{\epsilon}{K}$ denote the set of all ends $\beta$ of $X$ such that $\epsilon(K) = \beta(K)$.
	We refer to $\enbd{\epsilon}{K}$ as a \defword{basic end-space neighborhood} of $\epsilon$.
\end{definition}

Equip $\E{X}$ with the topology generated by
\[\{\enbd{\epsilon}{K}\,:\,\epsilon \in \E{X},\,K\subset X\tn{ is compact}\}\]
Note that if $\enbd{\epsilon}{K}$ and $\enbd{\eta}{L}$ both contain an end $\alpha$,
then they both contain ${\enbd{\alpha}{K \cup L}}$.
Thus, this collection indeed forms a basis for a topology on $\E{X}$.

\begin{lemma} \label{lem:endSpaceProperties}
	The space of ends $\E{X}$ is Hausdorff, compact, separable, and totally disconnected.
\end{lemma}

\begin{proof}
Fix a compact exhaustion $(K_i)$ of $X$.
If $\epsilon$, $\eta$ are two distinct ends of $X$, then there exists an $n$ such that $\epsilon(K_n) \neq \eta(K_n)$. $\enbd{\epsilon}{K_n}$ and $\enbd{\eta}{K_n}$ are two disjoint open sets containing $\epsilon$ and $\eta$, respectively. Thus, $\E{X}$ is Hausdorff. Furthermore, the complement of $\enbd{\epsilon}{K_n}$ is the union of $\enbd{\rho}{K_n}$ for all ends $\rho$ for which $\rho(K_n) \neq \epsilon(K_n)$. Thus, the complement of $\enbd{\epsilon}{K_n}$ is open, so $\enbd{\epsilon}{K_n}$ is clopen. As any two ends are separated by clopen subsets, $\E{X}$ is totally disconnected.

For each fixed $n$, there are finitely many unbounded components of $X \setminus K_n$, and so finitely many distinct sets $\enbd{\epsilon}{K_n}$ for $\epsilon \in \E{X}$. Because sets of the form $\enbd{\epsilon}{K_n}$ form a basis for the topology of $\E{X}$, and there are a finite number of these sets for each $n$, $\E{X}$ is second-countable, hence separable.

Lastly, we prove that $\E{X}$ is compact. Suppose for contradiction that $\E{X}$ has an open cover $\mathcal{C}$ with no finite subcover.
Say that a subset of $\E{X}$ is not finitely coverable if no finite subset of $\mathcal{C}$ covers it.
Suppose that $\{\epsilon\,:\,\epsilon(K_n) = V\}$ is not finitely coverable for some $V\subset X$.
Then, let $W_1, W_2, \ldots, W_k$ be the unbounded connected components of $X \setminus K_{n+1}$ contained in $V$.
We can see that $\{\epsilon\,:\,\epsilon(K_{n})=V\}$ is the union of $\{\epsilon\,:\,\epsilon(K_{n+1})=W_i\}$ over all $W_i$.
As $\{\epsilon\,:\,\epsilon(K_{n})=V\}$ is not finitely coverable, there exists some $j$ such that $\{\epsilon\,:\,\epsilon(K_{n+1})=W_j\}$ is not finitely coverable.
As $\E{X}$ is not finitely coverable, we obtain a sequence $W_1, W_2, \ldots $
where $W_i$ is a connected component of $X \setminus K_i$, $\{\epsilon\,:\,\epsilon(K_{i})=W_i\}$ is not finitely coverable, and $W_i \supset W_{i+1}$.
This determines an end, $\alpha$, of $X$. The end $\alpha$ is in some open subset $U \in \mathcal{C}$, which contains a subset of the form $\enbd{\alpha}{K_n}$ for some $n$. But, then, $\enbd{\alpha}{K_n}$ would be finitely coverable, which we know is not the case. This is a contradiction, so $\E{X}$ is compact.
\end{proof}

As ends are defined using compact subspaces, the theory of ends naturally utilizes \textit{proper maps}.
Recall that a continuous function $f: X \to Y$ is \defword{proper} provided the inverse image of each compactum is compact.
Topological spaces and proper maps form a category that is well-suited to the study of ends.
Given a proper map $f: X \to Y$, we may extend $f$ to a (proper) map $\E{f}: \E{X} \rightarrow \E{Y}$.
In other words, $\E{-}$ is a functor.

\begin{definition} \label{def:eMap}
Let $\E{f}(\epsilon)$ denote the unique end $\eta \in \E{Y}$ such that $f^{-1}(\eta(K)) \supset \epsilon(f^{-1}(K))$
for all compact $K \subset Y$.
\end{definition}

It may be shown that $\E{f}$ is well-defined and continuous.
For example, let $\R_+ = [0,\infty)\subset\R$.
A \defword{ray} in $X$ is a proper embedding $r:\R_+\to X$. 
There is a simple description of $\E{r}$.
Let $\epsilon$ be the unique end of $\R_+$, and let $\eta = \E{r}(\epsilon)\in\E{X}$.
Then, $\eta(K)$ is the unique component of $X \setminus K$ which contains $r((a, \infty))$ for some $a \in \R_+$.
In this case, we say that $r$ \defword{points to} $\eta$.

Although arguments are cleaner for connected spaces,
the theory of ends applies to \textit{disconnected} topological spaces that are otherwise nice for ends.
This allows us to define the ends of any manifold, which we will use frequently when adding $1$-handles at infinity.

\begin{definition}
Let $X$ be Hausdorff, locally compact, locally connected, and $\sigma$-compact, but not necessarily connected.
An \defword{end} of $X$ is an end of a connected component of $X$.
\end{definition}

When $X$ is disconnected, we may still define a natural topology on $\E{X}$ as above.
As a topological space, $\E{X}$ is the disjoint union of $\E{C}$ over all connected components $C$ of $X$.
If $X$ has only finitely many connected components, then $\E{X}$ is still compact.
If $X$ has infinitely many components, then $\E{X}$ may be noncompact.
Our focus is the effect of adding a single 1-handle at infinity to a manifold $M$.
This operation involves either one or two components of $M$, and any other components may be safely ignored.

\section{End Sum and 1-handles at Infinity} \label{sec:endSumIntroduction}

The basic idea of the end sum operation is to combine two noncompact manifolds of the same dimension along a proper ray in each.
For example, the interior of a boundary sum of two manifolds with boundary is an end sum of their interiors.
In general, end sum is more complicated than boundary sum as it applies to noncompact manifolds
that are not necessarily the interior of any compact manifold, and it requires a choice of ray in each manifold.
End sum also requires a choice of a tubular neighborhood map for each ray, though different choices yield isomorphic manifolds.
End sum is a special case of the more general operation of attaching a $1$-handle at infinity to a possibly disconnected manifold
(see also~\cite[pp.~1303--1305]{cg}).
We define these operations simultaneously for \CAT=\TOP, \PL, and \DIFF.
For the remainder of this section, embeddings and homeomorphisms are assumed to be \CAT.

Let $M$ be a (possibly disconnected) $n$-manifold.
A \defword{ray} in $M$ is a proper, locally flat embedding $r: \R_+ \to M$ with image in the manifold interior of $M$.
We often conflate a ray with its image in $M$.
Applying the end functor $\mathcal{E}$ to $r$ (see Definition~\ref{def:eMap}),
we have that $\E{r}$ picks out a single end $\eta \in \E{M}$.
In this case, we say that $r$ \defword{points to} $\eta$.
Define a \defword{tubular neighborhood map} of $r$ to be a proper, locally flat embedding
$\nu: [0, \infty) \times D^{n-1} \to M$ with image in the manifold interior of $M$ and such that $\nu(x, 0) = r(x)$.

Let $r$ and $r'$ be disjoint rays in $M$,
and let $\nu$ and $\nu'$ be tubular neighborhood maps of $r$ and $r'$ (respectively) with disjoint images.
Let $U \subset M$ be the image of $\nu$ restricted to the manifold interior of its domain.
Similarly, define $U'$ for $\nu'$.
As $U$ and $U'$ are locally flat embeddings of $n$-manifolds without boundary into the manifold interior of $M$,
$U$ and $U'$ are open subsets of the manifold interior of $M$ (hence, also of $M$).

Let $E = \R \times \R^{n-1}$, which is the $1$-handle at infinity that will be attached to $M$.
As in Figure~\ref{fig:endsum_diagram}, we partition $E$ into $E_- = (-\infty, 0) \times \R^{n-1}$,
$E_0 = \{0\} \times \R^{n-1}$, and $E_+ = (0, \infty) \times \R^{n-1}$.
\begin{figure}
    \centerline{\includegraphics[scale=0.7]{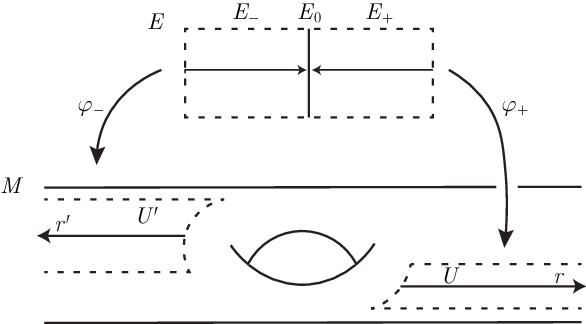}}
    \caption{Data for the addition of a 1-handle at infinity to $M$.}
	\label{fig:endsum_diagram}
\end{figure}
The restriction of $\nu$ to $U$ gives a homeomorphism $(0, \infty) \times \intr{(D^{n-1})} \cong U$
where $\intr{(D^{n-1})}$ denotes the manifold interior of the disk.
By reversing the first factor and radially expanding the second,
we get a homeomorphism $\phi_+: E_+ \rightarrow U$.
Similarly, the restriction of $\nu'$ to $U'$ gives a homeomorphism $(0, \infty) \times \intr{(D^{n-1})} \cong U'$.
By reversing \textit{and negating} the first factor and radially expanding the second, we get a homeomorphism $\phi_-: E_-  \rightarrow U'$.

We glue $E$ to $M$ using the maps $\phi_+$ and $\phi_-$.
Explicitly, let $\sim$ be the equivalence relation on $M \sqcup E$ generated
by $x \sim \phi_+(x)$ for $x \in E_+$ and $x \sim \phi_-(x)$ for $x \in E_-$.
Let $N$ be the quotient space of $M \sqcup E$ by this equivalence relation as in Figure~\ref{fig:endsum_final}.
\begin{figure}
    \centerline{\includegraphics[scale=0.65]{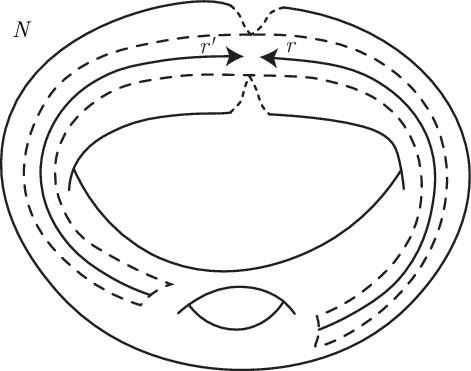}}
    \caption{The result $N$ of adding a 1-handle at infinity to $M$.}
	\label{fig:endsum_final}
\end{figure}
Let $q:M \sqcup E \to N$ be the associated quotient map.
We call $N$ a \defword{result of attaching a 1-handle at infinity} to $M$.
If $M=M_1 \sqcup M_2$ has exactly two components $M_1$ and $M_2$, each containing one of the rays $r$ or $r'$,
then we call $N$ an \defword{end sum} of $M_1$ and $M_2$ (or simply an end sum of $M$).
When $M$ is an oriented manifold and $\phi_+$ and $\phi_-$ are both orientation-preserving,
we call this construction an \defword{oriented $1$-handle}.

Below, we prove that $N$ is naturally an $n$-manifold
(we are not aware of a proof of this fact in the literature).
Define the composite maps $i_M:M \hookrightarrow M \sqcup E \stackrel{q}{\rightarrow} N$ and
$i_E:E \hookrightarrow M \sqcup E \stackrel{q}{\rightarrow} N$.
As the tubular neighborhood maps $\nu$ and $\nu'$ have disjoint images,
the maps $i_M$ and $i_E$ are injective.

\begin{lemma} \label{qopen}
	The quotient map $q:M \sqcup E \to N$ is open.
\end{lemma}

\begin{proof}
The quotient map is open if and only if the saturation of any open set in $M \sqcup E$ with respect to $\sim$ is still open in $M \sqcup E$.
Let $V$ be open in $M \sqcup E$ and let $W=q^{-1}(q(V))$ be the saturation of $V$ with respect to $\sim$.
We have:
\[W = V \cup \phi_+(V \cap E_+) \cup \phi_+^{-1}(V \cap U) \cup \phi_-(V \cap E_-) \cup \phi_-^{-1}(V \cap U')\]
As $\phi_+$ and $\phi_-$ are homeomorphisms between open subspaces of $M \sqcup E$, $W$ is open.
\end{proof}

\begin{corollary} \label{compsopen}
	The maps $i_M:M\to N$ and $i_E:E\to N$ are open.
\end{corollary}

\begin{proof}
The inclusions $M\hookrightarrow M \sqcup E$ and $E\hookrightarrow M \sqcup E$ are open by the definition of the disjoint union topology.
Now, apply Lemma~\ref{qopen}.
\end{proof}

The following is the key step where properness of the rays and tubular neighborhood maps are required.

\begin{lemma} \label{lem:endSumIsHausdorff}
	The quotient space $N$ is Hausdorff.
\end{lemma}

\begin{proof}
Let $a$ and $b$ be distinct points in $N$.
If $a, b \in \image\,i_M$, then $i_M^{-1}(a)$ and $i_M^{-1}(b)$ are separated by disjoint open sets $A$ and $B$ in $M$.
The sets $i_M(A)$ and $i_M(B)$ are disjoint since $i_M$ is injective, open in $N$ by Corollary~\ref{compsopen}, and thus separate $a$ and $b$.
If $a, b \in \image\,i_E$, then similarly $i_E^{-1}(a)$ and $i_E^{-1}(b)$ are separated by disjoint open sets in $E$
whose images under $i_E$ separate $a$ and $b$.
Lastly, consider the case where $a \in \image\,i_M$ and $b \in \image\,i_E$.
We may assume $b \in i_E(E_0)$, as otherwise $b \in \image\,i_M$.
Let $A$ be a bounded open subset of $M$ containing $a$.
As the tubular neighborhood maps $\nu$ and $\nu'$ are proper,
there exists some $s>0$ such that $\nu((s, \infty) \times D^{n-1})$ and $\nu'((s, \infty) \times D^{n-1})$ are disjoint from $A$.
So, there exists some $t>0$ such that $i_E((-t, t) \times \R^{n-1})$ is disjoint from $i_M(A)$.
The sets $i_E((-t, t) \times \R^{n-1})$ and $i_M(A)$ are open by Corollary~\ref{compsopen} and separate $a$ and $b$.
\end{proof}

\begin{lemma} \label{lem:endSumIsManifold}
	The quotient space $N$ is naturally an $n$-dimensional CAT manifold.
	If $M$ is oriented and the $1$-handle is oriented, then $N$ is naturally an oriented manifold.
	The maps $i_M$ and $i_E$ are open embeddings, and if the $1$-handle is oriented, then $i_M$ and $i_E$ are orientation-preserving.
\end{lemma}

\begin{proof}
The quotient space $N$ is Hausdorff by Lemma~\ref{lem:endSumIsHausdorff}.
As $M \sqcup E$ is second-countable and the open quotient of a second-countable space is second-countable, $N$ is also separable.
Every point of $N$ has a neighborhood homeomorphic to a subset of $M$ or $E$, so in particular, $N$ is locally Euclidean.
Thus, $N$ is a \TOP\ manifold.
If \CAT=\PL\ or \DIFF, then $N$ inherits a natural \PL\ or \DIFF\ structure using an atlas generated by the charts of $M$ and $E$.

If in addition $M$ is an oriented manifold and the 1-handle is oriented,
then the atlas generated by the oriented charts of $M$ and $E$ defines an orientation of $N$.

Lastly, $i_M$ and $i_E$ are injective, open by Corollary~\ref{compsopen}, and thus are open embeddings.
If in addition $M$ is an oriented manifold and the 1-handle is oriented, then $i_M$ and $i_E$ respect the orientation of $N$.
\end{proof}

There is an equivalent operation to adding a $1$-handle at infinity that is sometimes useful. 
In this alternative formulation, disjoint closed regular neighborhoods $A$ and $A'$ of $r$ and $r'$ are chosen in the interior of $M$.
Then, $N$ is defined by removing the manifold interiors of $A$ and $A'$ from $M$
and then gluing together the resulting boundary components---both are copies of $\R^{n-1}$---by an orientation reversing homeomorphism.
For more on this approach, see \cite[pp.~1813--1818]{cks}.

\section{Classification of Surfaces with Compact Boundary}\label{sec:classification}

The classification of noncompact surfaces with compact boundary is essential for our end sum uniqueness results.
This contrasts with end sum uniqueness results in higher dimensions which rely on isotopy uniqueness of rays~\cite{cg}.
Manifolds in this section are assumed to be \PL\, although the results hold for \DIFF\, and \TOP.
Following Richards~\cite{richards}, we define the genus and parity of any surface with compact boundary and extend those concepts to end invariants.
Using those invariants, we present a classification theorem of surfaces with compact boundary.

\subsection{Compact Surfaces}
\label{subsec:compactsurfaces}

We begin with the well-known classification of compact surfaces. For a proof, see \cite[pp.~446--476]{munkrestop}. Given a compact surface $M$, let $\chi(M)$ denote its Euler characteristic, and let $b(M)$ denote the number of boundary components of $M$.

\begin{theorem}[Classification of compact surfaces]\label{thm:finite2Manifolds}
	Let $M$ be a compact, connected surface. If $M$ is orientable, then for some non-negative integer $g$, $M$ is isomorphic to a sphere with $g$ handles and $b(M)$ holes, and $\chi(M) = 2 - 2g - b(M)$. If $M$ is non-orientable, then for some positive integer $k$, $M$ is isomorphic to the sphere with $k$ cross-caps and $b(M)$ holes, and $\chi(M) = 2 - k - b(M)$.
\end{theorem}

It will be more convenient for us to classify surfaces by their genera. We use the following convention (sometimes called the ``generalized genus'') to extend the concept of genus to any compact surface. Given a compact surface $M$, let $\piz{M}$ denote the number of connected components of $M$.

\begin{definition} \label{def:finitegenus}
	Let $M$ be a compact surface. The \defword{genus} of $M$, denoted $g(M)$, is an integer or half-integer defined by the following.
	
	\[ g(M) := \piz{M} - \frac{b(M) + \chi(M)}{2} \]
\end{definition}

For example, consider the surfaces described in Theorem~\ref{thm:finite2Manifolds};
in the orientable case $g(M)=g$, and in the non-orientable case $g(M)=k/2$.
Observe that the genus is unchanged if a surface is punctured (meaning an open disk is removed).
Further, genus is additive over disjoint union since it is a linear combination of functions that are each additive over disjoint union. These properties make the genus a user-friendly alternative to the Euler characteristic.
In a sense, the genus measures the complexity of a surface. This idea is made concrete by the following.

\begin{theorem} \label{thm:finiteGenusIsSize}
If $N$ is a compact surface, then $g(N) \geq 0$. If $M$ is a subsurface of $N$, then $g(M) \leq g(N)$.
\end{theorem}

\begin{proof}
The classification of compact surfaces (Theorem~\ref{thm:finite2Manifolds}) implies that every compact, connected surface has non-negative genus.
As genus is additive over disjoint union, every compact surface has non-negative genus.
	
Now consider a compact surface $N$ with a compact subsurface $M$.
It will be convenient for $M$ to avoid the boundary of $N$.
So, we let $N'$ be the result of gluing an external collar $\bd{N} \times [0, 1]$ to $N$ along $\bd{N}$. Note that $N$ and $N'$ are isomorphic and that $M$ lies in the manifold interior of $N'$.
	
As $M$ is a subsurface of $N'$ and is disjoint from $\bd{N}$, we have $E = \closr{N' \setminus M}$ is a subsurface of $N'$.
Note that $N'$ is obtained by gluing $E$ and $M$ together along their boundaries.
More precisely, each boundary component of $M$ is glued to a unique boundary component of $E$. We have
\begin{align*}
	   b(N') &= b(E) - b(M)\\
	   \chi(N') &= \chi(M) + \chi(E) - \chi(M \cap E) = \chi(M) + \chi(E)
\end{align*}
Beginning with $\piz{M} + \piz{E}$ components of $M$ and $E$, observe that each time a boundary component of $M$ is glued to one of $E$, the total number of components is reduced by $0$ or $1$. Thus, we have
	\begin{align*}
	    \piz{N'} &\geq \piz{M} + \piz{E} - b(M)
	\end{align*}
Therefore
	\begin{align*}
	    g(N') &= \piz{N'} - \frac{b(N') + \chi(N')}{2}\\
	    &\geq g(M) + g(E)
	\end{align*}
As $g(N') = g(N)$ and $g(E) \geq 0$, we get $g(N) \geq g(M)$.
\end{proof}

Define the \defword{parity} of a compact surface $M$, denoted $\parity{M}$, to be $2g(M)$ (mod 2). We call a compact surface \defword{even} if its parity is zero and \defword{odd} otherwise. Parity is strongly related to orientability by the following lemma.

\begin{lemma} \label{lem:parityIsOrientability}
If $N$ is a compact, orientable surface, then $\parity{N} = 0$.
If $N$ is a compact surface and $M \subset N$ is a subsurface such that $\bd{N} \subset M$ and $N \setminus M$ is orientable, then $\parity{N} = \parity{M}$ 
\end{lemma}

\begin{proof}
By the classification of compact surfaces (Theorem~\ref{thm:finite2Manifolds}), it can be checked that every compact, connected, orientable surface has even parity. As the parity of a disconnected surface is the sum of the parities of its connected components, every compact, orientable surface has parity zero.

Now, consider a compact surface $N$ with compact subsurface $M$ such that $\bd{N} \subset M$ and $N \setminus M$ is orientable. Let $E = \closr{N \setminus M}$. Because $\bd{N} \subset M$, $E$ is actually a (compact) subsurface of $N$. As $N \setminus M$ is orientable, $E$ is orientable. Thus, $E$ has even parity.

We have
\begin{align*}
	   b(N) &= b(M) - b(E)\\
	   \chi(N) &= \chi(M) + \chi(E) - \chi(M \cap E) = \chi(M) + \chi(E)
\end{align*}
Thus
\begin{align*}
	    g(N) - g(M) &= \left(\piz{N} - \frac{b(N) + \chi(N)}{2}\right) - \left(\piz{M} - \frac{b(M) + \chi(M)}{2}\right)\\
	    &= \piz{N} - \piz{M} + \frac{b(E)}{2} - \frac{\chi(E)}{2}\\
	    &= \piz{N} - \piz{M} - \piz{E} + b(E) + g(E)
\end{align*}
As $g(E)$ is an integer, $g(N) - g(M)$ is an integer. Hence, $\parity{N} = \parity{M}$.
\end{proof}

\subsection{Noncompact Surfaces with Compact Boundary}

To extend our definitions from compact surfaces to noncompact surfaces, we will require the existence of arbitrarily large compact subsurfaces. The key theorem we need for these arguments is the following.

\begin{theorem} \label{thm:subsurfaces}
Let $M$ be a surface and let $K \subset M$ be compact. Then there exists a compact, locally flatly embedded subsurface $S \subset M$ such that $K \subset S$.
\end{theorem}

\begin{proof}
Fix a triangulation $\Sigma$ of $M$. Let $L$ be the union of all triangles in $\Sigma$ that intersect $K$. Let $\Sigma'$ be the barycentric subdivision of $\Sigma$ and let $S$ be the union of all triangles in $\Sigma'$ that intersect $L$. It can be checked that $S$ is a \PL\ subsurface of $M$.
\end{proof}

As a consequence of Theorem~\ref{thm:finiteGenusIsSize}, we see that for a \textit{compact} surface $N$, $g(N)$ is the supremum of $g(M)$ over all compact subsurfaces $M$. This suggests a way to define the genus for arbitrary (noncompact) surfaces in terms of their compact subsurfaces. This generalized genus will play a vital role in our classification of noncompact surfaces with compact boundary.

\begin{definition} \label{def:genus}
	Let $N$ be a surface with compact boundary. The \defword{genus} of $N$, denoted $g(N)$, is the supremum of $g(M)$ over all
	compact subsurfaces $M \subset N$ (and may be infinite).
\end{definition}

Definition~\ref{def:genus} is not circular because the genus of an arbitrary surface is defined in terms of the genus of compact surfaces, which we have already defined. By Theorem~\ref{thm:subsurfaces}, we know that there exist arbitrarily large compact subsurfaces of $N$, so $g(N)$ is actually the \textit{limit} of $g(M)$ for compact subsurfaces $M$ as $M$ gets arbitrarily large.
Similar to the compact case, we may think of the genus of a noncompact surface as a measure of complexity
by the following extension of Theorem~\ref{thm:finiteGenusIsSize}.

\begin{theorem} \label{thm:genusissize}
If $N$ is a surface, then $g(N) \geq 0$. If $M$ is a subsurface of $N$, then $g(M) \leq g(N)$.
\end{theorem}

Parity applies to noncompact surfaces with compact boundary provided the surface is orientable outside of a compact subset. When $N$ is connected, this is equivalent to $N$ having no non-orientable ends (see Definition~\ref{def:endOrientable}).
If $N$ is orientable outside of a compact subset, then by Theorem~\ref{thm:subsurfaces}, there exists a compact subsurface $M \subset N$ such that $\bd{N} \subset M$ and $N \setminus M$ is orientable. Following Lemma~\ref{lem:parityIsOrientability}, we would like to define $\parity{N}$ to equal $\parity{M}$.

To prove this is well-defined, consider two compact subsurfaces $M_1, M_2 \subset N$ such that $N \setminus M_1$ and $N \setminus M_2$ are both orientable.
Let $S \subset N$ be a subsurface such that $M_1 \cup M_2 \subset S$.
By Lemma~\ref{lem:parityIsOrientability}, $\parity{M_1} = \parity{E}$ and $\parity{M_2} = \parity{E}$.
Thus, $\parity{M_1} = \parity{M_2}$. As a consequence, the following is well-defined.

\begin{definition}
Let $N$ be a surface with compact boundary that is orientable outside of a compact subset. The \defword{parity} of $N$, denoted $\parity{N}$, is the parity of any compact subsurface $M \subset N$ such that $\bd{M} \subset N$ and $N \setminus M$ is orientable.
\end{definition}

When $g(N)$ is finite, we can see that $\parity{N} \equiv 2g(N)$ (mod 2).
However, the parity of a surface may exist even when the genus is infinite.

\subsection{End-Invariants}
\label{subsec:endInvariants}

To continue our exploration of the properties of noncompact surfaces, we will need to classify the ends of a surface.
We will do this through the idea of an ``end-invariant'', a concept that can actually be applied to any topological space which is nice for ends.

Let $X$ and $X'$ be nice for ends. Define two ends $\eta \in \E{X}$ and $\eta' \in \E{X'}$ to be \defword{isomorphic} if there exist closed neighborhoods $A$ of $\eta$ and $A'$ of $\eta'$ and a homeomorphism $f: A \rightarrow A'$ such that $\E{f}(\eta) = \eta'$. It can be checked that isomorphism of ends is an equivalence relation.
An \defword{end-invariant} is any property or quantity associated to an end that is invariant under end isomorphism. 
One particularly important end-invariant is its orientability.

\begin{definition} \label{def:endOrientable}
An end $\eta$ of a manifold $M$ is \defword{orientable} if $\eta$ has a closed neighborhood $A$ such that $\intr{A}$ is orientable. Otherwise, we say that $\eta$ is \defword{non-orientable}.
\end{definition}

Examples of orientable and non-orientable ends are shown in Figure~\ref{fig:nonorientable}.
\begin{figure}[htb!]
\centering
\subfigure[Non-orientable end.]
{
    \label{fig:infnon}
    \includegraphics[scale=0.75]{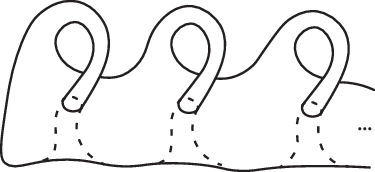}
}
\hspace{0.2cm}
\subfigure[Orientable end.]
{
    \label{fig:finnon}
    \includegraphics[scale=0.75]{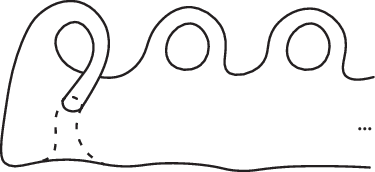}
}
\caption{One-ended, non-orientable surfaces.}
\label{fig:nonorientable}
\end{figure}
If $\eta$ has an orientable neighborhood $\eta(K)$ for some compact $K \subset M$, then every end in $\enbd{\eta}{K}$ is also orientable. Thus, the orientable ends of $M$ form an open subset of $\E{M}$.
Let $M$ be connected. If $M$ is orientable outside of a compact set, then all of its ends are orientable. Conversely, suppose that every end of $M$ is orientable. As every end of $M$ has an orientable neighborhood and $\E{M}$ is compact, $M$ is orientable outside of a compact set.

Another important end-invariant is the genus of an end.

\begin{definition}
	The \defword{genus} of an end $\eta$ of a surface $M$ with compact boundary is defined as the infimum of $g(\intr{A})$ over all closed neighborhoods $A$ of $\eta$.
\end{definition}

For example, any end collared by $S^1$ has genus zero.
If $A$ and $B$ are two neighborhoods of $\eta$ with $A \subset B$, then $g(\intr{A}) \leq g(\intr{B})$. Thus, the genus of an end can also be considered the limit of $g(U)$ for arbitrarily small neighborhoods $U$ of $\eta$. In fact, this limit can only have two possible values.

\begin{lemma}
	An end $\eta$ of a surface with compact boundary $M$ either has zero genus or infinite genus.
\end{lemma}
\begin{proof}
Let $\eta$ be an end of $M$ with finite genus $h > 0$. Let $K \subset M$ be compact such that $g(\eta(K)) = h$. By definition, there exists a compact subsurface $L \subset \eta(K)$ such that $g(L) = h$. Now, $g(\eta(K \cup L))$ must still be $h$, so there exists a compact subsurface $L' \subset \eta(K \cup L)$ such that $g(L')=h$. However, $L$ and $L'$ are disjoint, so $g(L \cup L') = g(L) + g(L') = 2h$. In addition, $L \cup L'$ is a subset of $\eta(K)$. Thus, $g(\eta(K)) \geq g(L \cup L') = 2h$. This is a contradiction.
\end{proof}

Let $M$ be a surface with compact boundary. If $\eta$ is genus zero, then $g(\eta(K)) = 0$ for some compact $K \subset M$. Thus, every end in $\enbd{\eta}{K}$ has genus zero. Thus, the genus zero ends of $M$ form an open subset of $\E{M}$.

Let $M$ be a connected surface with compact boundary. If $M$ has finite genus, then every end of $M$ has finite genus (hence zero genus). Conversely suppose that every end of $M$ has finite genus. As every end of $M$ has a genus zero neighborhood and $\E{M}$ is compact, $M$ is genus zero outside of a compact set.

The genus and orientability of an end are related. If an end $\eta$ of a surface has genus zero, then it has a genus zero neighborhood. As every genus zero surface is orientable, this neighborhood is also orientable. Thus, every genus zero end of a surface with compact boundary is also an orientable end.

Beyond these more geometric end-invariants, there are a whole host of algebraic end-invariants. The most important of these are the homotopy, homology, and cohomology groups at infinity \cite[pp.~229--281,~369--401]{geoghegan}. These are incredibly important, and provide very powerful tools for distinguishing various ends, and by extension, distinguishing various manifolds. Like the respective invariants of algebraic topology (that is, the homotopy, homology, and cohomology groups), the algebraic end-invariants listed so far are all functorial. But whereas the homotopy, homology, and cohomology groups are invariant under homotopy, their end-invariant counterparts are invariant only under proper homotopy.

The cohomology groups at infinity have proven particularly useful in investigating the uniqueness of end sums. In forthcoming work, Calcut and Guilbault prove that the cohomology groups at infinity of an end sum are independent of which choices are made in the end sum. However, the ring structure of the cohomology at infinity is crucial in work thus far to distinguish between non-isomorphic end sums of two manifolds. In some classifications of noncompact surfaces, the cohomology ring at infinity plays a crucial role, carrying effectively the same information as the space of ends and the geometric invariants of the ends \cite{goldman}.

\subsection{The Classification Theorem}
\label{subsec:classification}
We may now state the classification of surfaces without boundary. For a proof, see Richards' thesis \cite{richards}.

\begin{theorem} \label{thm:2manifolds}
	Let $M$ and $N$ be two connected surfaces without boundary with the same genus, the same orientability class, and homeomorphic end-space considered as a topological triplet $(A, B, C)$ where $A$ is the space of non-orientable ends, $B$ is the space of infinite genus ends, and $C$ is the space of ends. Then, $M$ and $N$ are isomorphic.
\end{theorem}

We also have a natural extension to surfaces with compact boundary.

\begin{corollary} \label{cor:2manifoldsPlus}
If $M$ and $N$ are two connected surfaces with compact boundary with the same invariants used in Theorem~\ref{thm:2manifolds}, and in addition $b(M) = b(N)$, then $M$ and $N$ are isomorphic.
\end{corollary}

\section{Main Theorem} \label{sec:mainTheorem}

Our main theorem is the following.

\begin{theorem} \label{thm:endSumUniquenessPLCase}
	Let $M$ be a (possibly disconnected) surface with compact boundary, and let $\eta$ and $\eta'$ be distinct ends of $M$.
	Let $N$ be a result of adding a 1-handle at infinity to $M$ along $\eta$ and $\eta'$.
	If $\eta$ and $\eta'$ are ends of distinct connected components of $M$, then $N$ is unique up to isomorphism.
	If $\eta$ and $\eta'$ are ends of the same connected, non-orientable component $M_1$ of $M$, then $N$ is unique up to isomorphism.
	Lastly, suppose that $\eta$ and $\eta'$ are ends of the same connected, orientable component $M_1$ of $M$.
	If the 1-handle is oriented, then  $N$ is unique up to isomorphism.
	If the 1-handle is not oriented, then $N$ is unique up to isomorphism.
\end{theorem}

This has the following immediate corollary for the end sum of surfaces.

\begin{corollary} \label{cor:endSumPL}
Let $M$ and $M'$ be distinct connected surfaces with compact boundary.
Let $\eta$ be an end of $M$ and let $\eta'$ be an end of $M'$.
Then, the end sum of $M$ and $M'$ along $\eta$ and $\eta'$ is uniquely determined up to isomorphism. 
\end{corollary}

The main tool we will use to prove Theorem~\ref{thm:endSumUniquenessPLCase} is the classification of surfaces with compact boundary (Corollary~\ref{cor:2manifoldsPlus}). That classification is based on the following characteristics of a connected surface $A$:
boundary, orientability, genus, parity, space of ends, subspace of infinite genus ends, and subspace of non-orientable ends.
In each of the next few subsections, we examine one or more of these characteristics and their behavior under the addition of a 1-handle at infinity.

\begin{conventions} \label{keyconventions}
For the remainder of this paper, unless otherwise stated, we will use the following conventions. Some of these conventions are reused from Section~\ref{sec:endSumIntroduction}.
\begin{itemize}
	\item All manifolds are PL.
	\item Let $M$ be an $n$-manifold, and let $\eta$ and $\eta'$ be ends of $M$ (not necessarily distinct).
	Let $r$ and $r'$ be disjoint rays in $M$ pointing to $\eta$ and $\eta'$ respectively.
	Let $\nu$ and $\nu': \R_+ \times D^{n-1} \rightarrow M$ be tubular neighborhood maps of $r$ and $r'$ respectively with disjoint images.
	\item Let $E = \R^n$, and define the subspaces $E_- = (-\infty, 0) \times \R^{n-1}$, $E_0 = \{0\} \times \R^{n-1}$, and $E_+ = (0, \infty) \times \R^{n-1}$.
	\item Let $N$ be the result of adding a 1-handle at infinity to $M$ along the tubular neighborhood maps $\nu$ and $\nu'$ by attaching the strip $E$ as described in Section~\ref{sec:endSumIntroduction}. Let $\phi_-: E_- \rightarrow M$ and $\phi_+: E_+ \rightarrow M$ be the gluing maps used in the construction of $N$. Let $i_M$ be the natural embedding of $M$ into $N$, and let $i_E$ be the natural embedding of $E$ into $N$.
	\item Let $R_m = [-m,m]^n$. We will use $(R_m)$ as a compact exhaustion of $E$.
	\item $(K_m)$ is a compact exhaustion of $M$ as specified in Theorem~\ref{thm:steppingStone}.
	\item Let $L_m = i_M(K_m) \cup i_E(R_m)$. We will use $(L_m)$ as a compact exhaustion of $N$.
\end{itemize}
\end{conventions}

In our study of the ends of $M$ and $N$, it will be helpful to have exhaustions that are situated nicely with respect to both the rays $r$ and $r'$ and the tubular neighborhood maps $\nu$ and $\nu'$.

\begin{lemma} \label{lem:compactSteppingStone}
	For every compact subset $K \subset M$, there exists a compact $n$-manifold $A \subset M$ such that:
	(i) $K \subset \intr{A}$,
	(ii) $\nu^{-1}(A) = [0, a] \times D^{n-1}$, and
	(iii) $\nu'^{-1}(A) = [0, a'] \times D^{n-1}$ for some $a, a' > 1$. This situation is presented in Figure~\ref{fig:nicePLExhaustion}.
\end{lemma}

\begin{figure}[htb!]
	\begin{center}
		\includegraphics[scale=0.85]{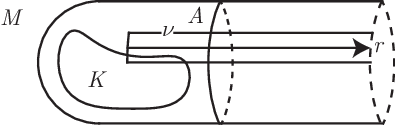}
		\caption{Surface $M$ containing a compact set $K$, ray $r$, and tubular neighborhood $\nu$ of $r$.
			The compact subsurface $A$ of $M$ contains $K$ in its interior and meets $\nu$ nicely.}
		\label{fig:nicePLExhaustion}
	\end{center}
\end{figure}

\begin{proof}
Pick $c$ such that $\nu^{-1}(K) \subset [0, c] \times D^{n-1}$. Similarly for $c'$ and $\nu'$. Set
\[
	L = K \cup \nu([0,c] \times D^{n-1}) \cup \nu'([0, c'] \times D^{n-1})
\]
Let $T$ be a regular neighborhood of $L$ in $M$
(see Rourke and Sanderson~\cite[Ch.~3]{rourkesanderson} and Scott~\cite{scott} for the \PL theory of regular neighborhoods).
Importantly for us, $K \subset \intr{T}$ and $T$ is a compact $n$-dimensional submanifold of $M$ \cite[p.~34]{rourkesanderson}.

Construct an ambient isomorphism $\psi$ with support disjoint from $L$ and $\image{\nu'}$
such that $\psi(\bd{T})$ is transverse to $\bd{\image{\nu}}$ \cite[p.~185]{az}.
Replace $T$ by $\psi(T)$.
Note that $T$ is still a regular neighborhood of $L$, but now intersects $\image{\nu}$ transversely.
Similarly, modify $T$ to be transverse to $\image{\nu'}$ as well.

Pick $d > 1$ such that $\nu^{-1}(T)$ is contained in $[0, d] \times D^{n-1}$. Similarly, find $d'$. Set
\[
	A = T \cup \nu([0,d] \times D^{n-1}) \cup \nu'([0,d'] \times D^{n-1})
\]
It is clear that $A$ is locally Euclidean at all points except $\bd{T} \cap \bd{\image{\nu}}$ and $\bd{T} \cap \bd{\image{\nu'}}$. By transversality, for every $p \in \bd{T} \cap \bd{\image{\nu}}$, there exists a neighborhood $U$ of $p$ and a coordinate map $\phi: U \rightarrow \R^n$ where $\phi(p) = 0$, $\phi(T) = [0, \infty) \times \R \times \R^{n-2}$, and $\phi(\image{\nu}) = \R \times [0, \infty) \times \R^{n-2}$. In this local coordinate map, it is clear that $A$ is locally Euclidean at $p$. Thus, $A$ is a $n$-dimensional submanifold of $M$.
Since $K \subset \intr{T}$, $K \subset \intr{A}$. By construction, $\nu^{-1}(A) = [0,d] \times D^{n-1}$ and $\nu'^{-1}(A) = [0,d'] \times D^{n-1}$.
\end{proof}

%

By constructing larger and larger submanifolds using Lemma~\ref{lem:compactSteppingStone}, we can create a nice compact exhaustion for $M$.

\begin{theorem} \label{thm:steppingStone}
	Let $C \subset M$ be compact.
	There exists a compact exhaustion $(K_m)$ of $M$ such that
	(i) $C \subset K_m$ for all $m$,
	(ii) $K_m$ is an $n$-manifold for all $m$,
	(iii) $\nu^{-1}(K_m) = [0, a_m] \times D^{n-1}$ and $\nu'^{-1}(K_m) = [0, a_m'] \times D^{n-1}$ for some sequences $a_m, a_m' > 1$.
\end{theorem}

\begin{proof}
Start with any compact exhaustion $(A_m)$ of $M$. Inductively build $(K_m)$ as follows: set $K_0 = C$. For all $m > 0$, let $K_m$ be a compact submanifold of $M$ containing $A_m \cup K_{m-1}$ that has nice intersection with $\nu$ and $\nu'$, as given by Lemma~\ref{lem:compactSteppingStone}.
\end{proof}

Using a compact exhaustion of $M$ provided by Theorem~\ref{thm:steppingStone}, we may construct a nice compact exhaustion of $N$ as well. Define $L_m = R_m \cup K_m$. Because $K_m$ and $R_m$ have transverse intersection, $(L_m)$ is a compact exhaustion of $N$ by $n$-manifolds. In this section, we use the compact exhaustions: $(K_m)$ of $M$, $(R_m)$ of $E$, and $(L_m)$ of $N$. These exhaustions are depicted in Figure~\ref{fig:compactexhaustion}.

\begin{figure}[htb!]
	\centering
	\subfigure[Manifolds $M$ and $E$ with compact exhaustions $(K_m)$ and $(R_m)$ respectively]
	{
	    \label{fig:ceinput}
	    \includegraphics[scale=0.70]{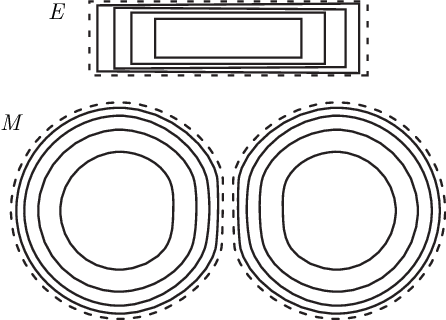}
	}
	\hspace{0.1cm}
	\subfigure[End sum $N$ with compact exhaustion $(L_m)$. The image of $E$, $i_E(E)$, is dashed for clarity]
	{
	    \label{fig:ceoutput}
	    \includegraphics[scale=0.70]{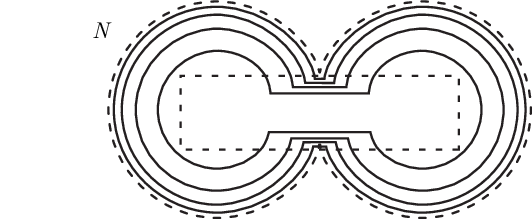}
	}
	\caption{Constructing the compact exhaustion $(L_m)$ of the end sum $N$.}
	\label{fig:compactexhaustion}
\end{figure}

\subsection{Space of Ends} \label{subsec:endSumEnds}

In this section, we tackle the end-space of $N$. In the process, we will encounter ordinary and extraordinary ends of $N$, and the map $\eh: \E{N} \rightarrow \E{M} / (\eta \equiv \eta')$.

Call an end $\alpha$ of $N$ \defword{ordinary} provided that it has a neighborhood contained in $i_M(M)$, or equivalently, that it has a neighborhood disjoint from $i_E(E_0)$. Otherwise, call $\alpha$ \defword{extraordinary}. An ordinary end is essentially not involved in the 1-handle construction, and is the simplest to describe.

\begin{lemma} \label{lem:componentsOfComplementOrdinary}
If $\alpha$ is an ordinary end of $N$, then $\alpha(i_M(K))$ is disjoint from $i_E(E_0)$ for sufficiently large compact $K \subset M$.
\end{lemma}

\begin{proof}
Because $\alpha$ is ordinary, $\alpha(L_m)$ is disjoint from $i_E(E_0)$ for some $m$. Consider $U_a = i_E((-a, a) \times \R^{n-1})$. For small enough $a$, $L_m \cap U_a = i_E((-a, a) \times [-m,m]^{n-1})$ by construction of $L_m$. If so, then all components of $U_a \setminus L_m$ intersect $i_E(E_0)$. So, $\alpha(L_m)$ is disjoint from $U_a$, and therefore $\alpha(L_m)$ is a connected component of $N \setminus (L_m \setminus U_a)$. Thus, $\alpha(L_m \setminus U_a) = \alpha(L_m)$ is disjoint from $i_E(E_0)$.
\end{proof}

\begin{lemma} \label{lem:componentsOfComplementExtraordinary}
If $\alpha$ is an extraordinary end, then $\alpha(i_M(K))$ is the disjoint union $i_M(\eta(K)) \cup i_M(\eta'(K)) \cup i_E(E_0)$ for all compact $K \subset M$.
\end{lemma}
\begin{proof}
Since $\alpha$ is extraordinary, $\alpha(i_M(K))$ must not be disjoint from $i_E(E_0)$. Thus, $\alpha(i_M(K))$ contains $i_E(E_0)$. Every neighborhood of $i_E(E_0)$ intersects $i_M(\eta(K))$ and $i_M(\eta'(K))$. Furthermore, $i_M(\eta(K)) \cup i_M(\eta'(K)) \cup i_E(E_0)$ is an open set. Its complement in $N \setminus i_M(K)$ is a union of open sets of the form $i_M(U)$, where $U$ is a component of $M \setminus K$. It follows that $i_M(\eta(K)) \cup i_M(\eta'(K)) \cup i_E(E_0)$ is a component of $N \setminus i_M(K)$, hence equals $\alpha(i_M(K))$.
\end{proof}

\begin{lemma} \label{lem:ordinaryEnds}
To every ordinary end $\alpha$ of $N$, we can assign a unique end $\beta$ of $M$ such that $i_M(\beta(K)) \subset \alpha(i_M(K))$ for all compact $K \subset M$. Furthermore, $\alpha$ and $\beta$ are isomorphic as ends, and $\beta \neq \eta, \eta'$.
\end{lemma}

\begin{proof}
By Lemma~\ref{lem:componentsOfComplementOrdinary}, $\alpha(i_M(K_m))$ is disjoint from $i_E(E_0)$ for some $m$. For any compact $A \subset M$, consider $L_A = i_M(A \cup K_m)$. This is a compact subset of $N$, and $\alpha(L_A)$ is disjoint from $i_E(E_0)$. Define $\beta(A)$ to be the unique connected component of $M \setminus A$ containing $i_M^{-1}(\alpha(L_A))$. Note that $\beta$ is a well-defined end of $M$.
So, $i_M(\beta(A))$ is the connected component of $i_M(M) \setminus i_M(A)$ containing $\alpha(L_A)$, and $\alpha(i_M(A))$ is the connected component of $N \setminus i_M(A)$ containing $\alpha(L_A)$. Thus, $i_M(\beta(A)) \subset \alpha(i_M(A))$.

Suppose $\beta'$ is another end satisfying that condition. For large enough compact $K \subset M$, $\alpha(i_M(K))$ is disjoint from $i_E(E_0)$, so $i_M(\beta(K)) = \alpha(i_M(K)) = i_M(\beta'(K))$. Thus, $\beta' = \beta$.

For large enough compact $K \subset M$, $i_M(\beta(K)) = \alpha(i_M(K))$. Thus, $\alpha$ and $\beta$ are isomorphic as ends.
\end{proof}

When $\alpha$ is an ordinary end of $N$, let $\eh(\alpha)$ denote the associated end of $M$, as given by Lemma~\ref{lem:ordinaryEnds}. Note that with this definition, $\eh$ is a function from the ordinary ends of $N$ to $\E{M} \setminus \{\eta, \eta'\}$. Extend $\eh$ to a map from $\E{N}$ to $\E{M} / (\eta \equiv \eta')$ by setting $\eh(\alpha) = \eta \equiv \eta'$ whenever $\alpha$ is an extraordinary end of $N$. By $\E{M} / (\eta \equiv \eta')$, we mean the topological quotient of $\E{M}$ by the equivalence relation $\equiv$ generated by $\eta \equiv \eta'$. This quotient map is closed (since it is a continuous function from a compact space to a Hausdorff space), but is not necessarily open (consider two copies of the surface $M$ above in Figure~\ref{fig:oddex1} and end sum along their nonisolated ends). It will be shown in Lemma~\ref{lem:partEndSumEnds} that in many cases, this $\eh$ is also a homeomorphism, completely describing the end space of $N$.

%
%
%
%

\begin{lemma} \label{lem:partEndSumEnds}
The canonical map $\eh: \E{N} \rightarrow \E{M} / (\eta \equiv \eta')$ is continuous. If $n \geq 2$, then $h$ is surjective. If (i) $n = 2$, $\bd{M}$ is compact, and $\eta \neq \eta'$, or (ii) $n \geq 3$, then $\eh$ is a homeomorphism.
\end{lemma}
\begin{proof}
First, we prove that $h$ is continuous. Let $\alpha$ be an end of $N$, and let $\beta = \eh(\alpha)$. Let $U$ be an open neighborhood of $\beta$ in $\E{M}/(\eta \equiv \eta')$. We will show that $\eh(V) \subset U$ for some open neighborhood $V$ of $\alpha$.
If $\alpha$ is ordinary, then we may take $U = \enbd{\beta}{K}$ for some compact $K \subset M$, and by Lemma~\ref{lem:componentsOfComplementOrdinary}, we can assume that $\alpha(i_M(K))$ is disjoint from $i_E(E_0)$. For every end $\tau \in \enbd{\alpha}{i_M(K)}$, $\tau(K)$ is disjoint from $i_E(E_0)$, so $\tau$ is ordinary. In fact, $\eh(\tau)(K) = \beta(K)$, so $\eh(\tau) \in U$.
If $\alpha$ is extraordinary, then we may take $U = \enbd{\eta}{K} \cup \enbd{\eta'}{K}$ for some compact $K \subset M$. For every extraordinary end $\tau \in \enbd{\alpha}{i_M(K)}$, $\eh(\tau) = \beta \in U$. For every ordinary end $\tau \in \enbd{\alpha}{i_M(K)}$,
\[
	i_M(\eh(\tau)(K)) \subset \tau(i_M(K)) = \alpha(i_M(K)) = i_M(\eta(K)) \cup i_M(\eta'(K)) \cup i_E(E_0)
\]
The last equality is true by Lemma~\ref{lem:componentsOfComplementExtraordinary}. Thus, $\eh(\tau) = \eta(K)$ or $\eta'(K)$. So, $\eh(\tau) \in U$. Thus, $\eh$ is continuous.

In the other direction, let $\beta \neq \eta, \eta'$ be an end of $M$. Pick compact $K \subset M$ such that $\beta(K) \neq \eta(K), \eta'(K)$. Note that $i_M(\beta(K))$ is also a connected component of $N \setminus i_M(K)$. Given compact $L \subset N$, set $L' = K \cup (\beta(K) \cap i_M^{-1}(L))$. Note that $L'$ is a compact subset of $M$. Define $\alpha(L)$ as the unique component of $N \setminus L$ containing $i_M(\beta(L'))$. Then, $\alpha$ is the unique end of $M$ such that $\eh(\alpha) = \beta$.

The map $\eh$ is surjective provided there is at least one extraordinary end of $N$. Furthermore, as a continuous map between compact spaces, $\eh$ is a homeomorphism as long as it is bijective, that is, as long as there is one extraordinary end of $N$. An end $\alpha$ of $N$ is extraordinary provided that $\alpha(L)$ intersects $i_E(E_0)$ for all compact $L \subset N$. If $n \geq 2$, then $i_E(E_0)$ is noncompact, so it has at least one end $\rho$. Let $\alpha(L)$ be the unique component of $N \setminus L$ containing $\rho(i_E(E_0) \cap L)$. Then $\alpha$ is an extraordinary end of $N$.

If $n \geq 3$, then $i_E(E_0) \cong \R^{n-1}$ has a single end $\rho$, and for any extraordinary end $\zeta$, $\zeta(L) \supset \rho(i_E(E_0) \cap L)$. Thus, $N$ has exactly one extraordinary end.

If $n = 2$, then $i_E(E_0) \cong \R$ has two ends, $\rho_-$ and $\rho_+$. For any extraordinary end $\zeta$ of $N$, $\zeta(L)$ contains at least one of $\rho_-(i_E(E_0) \cap L)$ or $\rho_+(i_E(E_0) \cap L)$. Thus, $N$ has at most two extraordinary ends, and $N$ has a single extraordinary end precisely when $\rho_-(i_E(E_0) \cap L)$ and $\rho_+(i_E(E_0) \cap L)$ are always contained in the same connected component of $N \setminus L$.

Suppose that $n = 2$, $\eta \neq \eta'$, and $\bd{M}$ is compact.
It suffices to prove that $\rho_+(i_E(E_0) \cap L_m)$ and $\rho_-(i_E(E_0) \cap L_m)$ are connected in $N \setminus L_m$ for large enough $m$.
Starting at $\rho_+(i_E(E_0) \cap L_m)$, trace along the boundary of $i_E(R_m)$ until it intersects $i_M(K_m)$.
Let $T$ be the boundary component of $K_m$ that we intersect.
Since $\bd{M}$ is compact, $\bd{M} \subset K_m$ for large enough $m$.
Thus, staying within $N \setminus L_m$, we can traverse $i_M(T)$ until we again intersect with $i_E(R_m)$.
Traverse the boundary of $i_E(R_m)$ until we return to $\rho_+(i_E(E_0) \cap L_m)$ or arrive at $\rho_-(i_E(E_0) \cap L_m)$.
If we return to $\rho_+(i_E(E_0) \cap L_m)$, then $T$ must be a boundary component of both $\eta(K_m)$ and $\eta'(K_m)$.
Since $\eta \neq \eta'$, $\eta(K_m) \neq \eta'(K_m)$ for large enough $m$, in which case this situation is impossible.
We have traced a path from $\rho_+(i_E(E_0) \cap L_m)$ to $\rho_-(i_E(E_0) \cap L_m)$ within $N \setminus L_m$.
Thus, they are connected in $N \setminus L_m$, and $N$ has exactly one extraordinary end.
\end{proof}

In the future, whenever there is a unique extraordinary end of $N$, we will call it $\zeta$. While the ordinary ends of $N$ are isomorphic to their corresponding ends in $M$ and easy to classify, understanding $\zeta$ takes more effort.

\subsection{Boundary} \label{subsec:endSumBoundary}

In this section, we briefly consider the boundary of $N$.

\begin{theorem} \label{thm:boundaryOfTheEndSum}
There is a canonical isomorphism $\bd{N} \cong \bd{M}$.
\end{theorem}

\begin{proof}
As $i_M$ and $i_E$ are open embeddings, $\bd{N} = i_M(\bd{M}) \cup i_E(\bd{E})$. As $E$ has no boundary components, $\bd{N} = i_M(\bd{M})$.
\end{proof}

\subsection{Orientability} \label{subsec:endSumOrientation}

In this section, we tackle the orientability of $N$.
The orientability of $N$ depends on the orientability of $M$, which ends are used for the 1-handle, and the orientation of the 1-handle.
For examples, see Figures~\ref{fig:r2exs} and~\ref{fig:annulusbasic} in Section~\ref{sec:introduction} above.
To prove that $N$ is orientable, it suffices to exhibit an orientation on $N$.
To prove that $N$ is non-orientable, we will use the notion of an orientation-reversing loop.

\begin{definition} \label{def:orientationPreservingLoop}
A loop $\gamma$ in a manifold $M$ is \defword{orientation-preserving} if $\gamma$ lifts to a loop in the oriented double-cover of $M$. Define two paths $\Gamma_1$ and $\Gamma_2$ in $M$ with the same endpoints to have the \defword{same effect on orientation} if the concatenation of $\Gamma_1$ by the reverse of $\Gamma_2$ is an orientation-preserving loop.
\end{definition}

For a description of the oriented double-cover of a manifold, see \cite[pp.~233--235]{hatcherbook}. As a homotopy of a loop in the base space lifts to a homotopy in the covering space, two homotopic loops are either both orientation-preserving or both orientation-reversing. Importantly for us, every loop in $M$ is orientation-preserving if and only if $M$ is orientable.

\begin{theorem} \label{thm:endSumOrientability}
If $M$ is connected, then $N$ is orientable if and only if $M$ is orientable and the 1-handle is oriented.
If $M = M_1 \cup M_2$ where each $M_i$ is a connected component of $M$, $\eta$ is an end of $M_1$ and $\eta'$ is an end of $M_2$, then $N$ is orientable if and only if both $M_1$ and $M_2$ are orientable.
\end{theorem}

\begin{proof}
First consider the case where $M$ is connected.
Suppose that $M$ is non-orientable.
As $i_M(M)$ is an open subset of $N$, $N$ is also non-orientable.
Next, suppose that $M$ is orientable. Fix an orientation of $M$. If the 1-handle is oriented, then by Lemma~\ref{lem:endSumIsManifold}, $N$ is orientable. If the 1-handle is not oriented, then we may assume without loss of generality that $\phi_+$ is orientation-preserving, and $\phi_-$ is orientation-reversing. Fix $a \in i_E(E_-)$ and $b \in i_E(E_+)$. Choose a path within $i_E(E)$ from $a$ to $b$ and a path within $i_M(M)$ from $b$ to $a$. Concatenating these paths gives an orientation-reversing loop. Thus, $N$ is non-orientable.

Second, consider the case where $M = M_1 \cup M_2$ where each $M_i$ is a connected component of $M$, $\eta$ is an end of $M_1$, and $\eta'$ is an end of $M_2$.
Suppose that $M_1$ is non-orientable. As $i_M(M_1)$ is an open subset of $N$, $N$ is also non-orientable. Similarly, if $M_2$ is non-orientable, then so is $N$.
Next, suppose that both $M_1$ and $M_2$ are orientable. By reversing the orientations on $M_1$ and $M_2$ if needed, we may assume that the 1-handle is oriented. By Lemma~\ref{lem:endSumIsManifold}, $N$ is orientable.
\end{proof}

\subsection{Orientability of the Ends} \label{subsec:endSumEndOrientation}

Thus far in Section~\ref{sec:mainTheorem}, we have given general results on 1-handles and end-sum for $n$-manifolds.
We now narrow our focus to the conditions of Theorem~\ref{thm:endSumUniquenessPLCase}.
From here on, we assume that $M$ is a surface with compact boundary and $\eta \neq \eta'$.
In particular, Lemma~\ref{lem:partEndSumEnds} implies that $N$ has a unique extraordinary end $\zeta$.

In this subsection, we study the orientability of the ends of $N$.
Recall that we define an end of $M$ to be orientable provided that it has an orientable neighborhood, and the set of orientable ends is open in the space of all ends.

\begin{theorem} \label{thm:endSumEndOrientations}
For every ordinary end $\alpha$ of $N$, $\eh(\alpha)$ has the same orientability as $\alpha$.
The extraordinary end $\zeta$ of $N$ is orientable if and only if both $\eta$ and $\eta'$ are orientable.
\end{theorem}

\begin{proof}
From Theorem~\ref{lem:ordinaryEnds}, we know that for every ordinary end $\alpha$ of $N$,
$\eh(\alpha)$ is isomorphic to $\alpha$ as ends.
Hence, $\eh(\alpha)$ and $\alpha$ have the same orientability.

Suppose first that $\eta$ is non-orientable. Let $m$ be large enough so that $\eta(K_m) \neq \eta'(K_m)$.
Let $\Gamma_m$ be an orientation-reversing loop in $\eta(K_m)$. From the construction of $R_m$ and $K_m$, we know that $K_m$ splits $R_m$ in two.
We may homotope any segment of $\Gamma_m$ that intersects (but does not cross) $R_m$ to one that is disjoint from $R_m$. We may replace any segment of $\Gamma_m$ that crosses $R_m$ by a path that follows $\bd{R_m}$ until it meets $\bd{K_m}$, then follows $\bd{K_m}$ until it meets $\bd{R_m}$ on the other side, then follows back up $\bd{R_m}$. Note that this path has the same effect on orientation as the original segment of $\Gamma_m$.
The modified loop created in this way will still be orientation-reversing, but will also be disjoint from $\phi_+(R_m)$.
Thus, $\eta(K_m) \setminus \phi_+(R_m)$ is non-orientable.
Thus, $\zeta(L_m)$ is non-orientable.
Since this is true for sufficiently large $m$, $\zeta$ is non-orientable.

Now suppose that both $\eta$ and $\eta'$ are orientable.
We can pick $m$ large enough that $\eta(K_m)$, $\eta'(K_m)$ are distinct and orientable.
Choose orientations on $\eta(K_m)$ and $\eta'(K_m)$ that are compatible with the 1-handle. These orientations define an orientation on $\zeta(i_M(K_m))$. Thus, $\zeta$ is orientable.
\end{proof}

\subsection{Genus-Related Properties} \label{subsec:genus}

Lastly, we study the genus-related properties of $N$. Specifically, we determine the genus of $N$, the parity of $N$ (if all ends are orientable), and which ends of $N$ have infinite genus.

\subsubsection{Genus and Parity} \label{subsubsec:endSumGenus}

Recall that we use $g(M)$ to denote the genus of $M$
and $\parity{M}$ to denote the parity of $M$ (as an element of $\mathbb{Z}/2\mathbb{Z}$).

\begin{theorem} \label{thm:endSumGenus}
	If $\eta$ and $\eta'$ lie in distinct components of $M$, then $g(N) = g(M)$.
	If $\eta$ and $\eta'$ lie in the same component of $M$, then $g(N) = g(M) + 1$.
	If $M$ is orientable outside of a compact set, then $N$ is orientable outside of a compact set and $\parity{N} = \parity{M}$.
\end{theorem}

\begin{proof}
We may assume that $N$ is connected.
If $M$ is orientable outside of a compact set, then $M$ has no non-orientable ends and
Theorem~\ref{thm:endSumOrientability} implies that $N$ also has no non-orientable ends.
Thus, the parity $\parity{N}$ is well-defined and equals the limit of $\parity{L_m}$ as $m\to\infty$.
On the other hand, the genus $g(M)$ is always well-defined (possibly infinite) and equals the limit of $g(L_m)$ as $m\to\infty$. We will compute the genus of $L_m$ and use that to compute the genus and parity of $N$.

Note that $R_m$ meets the boundary of $K_m$ in two intervals.
Since $\eta \neq \eta'$ and $\bd{M}$ is compact, these intervals lie in distinct components of $\bd{K_m}$ for all sufficiently large $m$.
If so, then $b(L_m) = b(K_m) - 1$.
For the Euler characteristic, we have
$\chi(L_m) = \chi(K_m) + \chi(R_m) - \chi(K_m \cap R_m) = \chi(K_m) - 1$.

If $M$ is connected, then for all sufficiently large $m$,
both components of $K_m \cap R_m$ lie in the same connected component of $K_m$
and so $\piz{L_m} = \piz{K_m}$. Thus
\begin{align*}
    g(L_m) &= \piz{L_m} - \frac{b(L_m)+\chi(L_m)}{2}\\
    &= \piz{K_m} - \frac{b(K_m)+\chi(K_m)}{2} +1\\
    &= g(K_m) +1
\end{align*}
In this case, $g(N) = g(M)+1$ and $\parity{N} = \parity{M}$ (if applicable).

If $M$ is disconnected, then one component of $K_m \cap R_m$ lies in each connected component of $M$
and so $\piz{L_m} = \piz{K_m} - 1$. Thus
\begin{align*}
    g(L_m) &= \piz{L_m} - \frac{b(L_m)+\chi(L_m)}{2}\\
    &= \piz{K_m} -1 - \frac{b(K_m)+\chi(K_m)}{2} + 1\\
    &= g(K_m)
\end{align*}
In this case, $g(N) = g(M)$ and $\parity{N} = \parity{M}$ (if applicable).
\end{proof}

\subsubsection{Genus of the Ends} \label{subsubsec:endSumEndsGenus}

Recall that the genus of an end $\tau \in \E{M}$ is the limit of $g(K_n)$ as $n \rightarrow \infty$ for any compact exhaustion $K_n$ of $M$. The genus of an end is either zero or infinity.

\begin{theorem}
	If $\alpha$ is an ordinary end of $N$, then $\alpha$ has the same genus as $\eh(\alpha)$.
	Furthermore, $\zeta$ has infinite genus if and only if either $\eta$ or $\eta'$ has infinite genus.
\end{theorem}

\begin{proof}
To simplify the argument, assume without loss of generality that $\eta(K_m) \neq \eta'(K_m)$ for all $m$.
If $\alpha$ is an ordinary end of $N$, then $\alpha$ and $\eh(\alpha)$ are isomorphic as ends.
Thus, $\alpha$ has the same genus as $\eh(\alpha)$.

Next, consider the extraordinary end $\zeta$.
To compute $g(\zeta)$, we introduce some subsurfaces of $N$ as depicted in Figure~\ref{fig:endsumgenusends}.

\begin{figure}[htb!]
    \centerline{\includegraphics[scale=0.8]{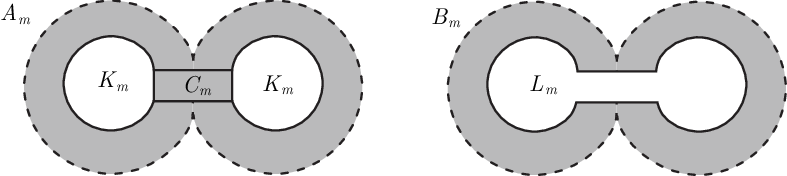}}
    \caption{Subsurfaces of $N$: $A_m$ (shaded at left) and $B_m$ (shaded at right).
		Here, $C_m\subset R_m$ is a $2$-disk, $A_m=B_m\cup C_m$, and
		the unshaded dumbbell (at right) is $L_m$.}
	\label{fig:endsumgenusends}
\end{figure}

Here, $A_m = \overline{\zeta\pa{K_m}}$,
$B_m = \overline{\zeta\pa{L_m}}$, and $A_m$ is the union of $B_m$ and a $2$-disk $C_m \subset R_m$.
As $A_m$ and $B_m$ are connected, we have $\piz{A_m} = \piz{B_m} = 1$.
Gluing $C_m$ onto $B_m$ separates one boundary component of $B_m$ into two,
so $b(A_m) = b(B_m) + 1$.
\[
\chi\pa{A_m} = \chi\pa{B_m} + \chi\pa{C_m} - \chi\pa{B_m \cap C_m} = \chi\pa{B_m} - 1
\]
It follows that $g(A_m) = g(B_m)$.

Note that $A_m$ is an end sum of $\overline{\eta\pa{K_m}}$ and $\overline{\eta'(K_m)}$.
By Theorem~\ref{thm:endSumGenus}, $g\pa{A_m} = g\pa{\overline{\eta\pa{K_m}}} + g\pa{\overline{\eta'\pa{K_m}}}$.
In the limit as $m \rightarrow \infty$, $g\pa{\eta\pa{K_m}}$ converges to $g\pa{\eta}$,
$g\pa{\eta'\pa{K_m}}$ converges to $g\pa{\eta'}$, and $g\pa{B_m}$ converges to $g\pa{\zeta}$.
Thus, $g\pa{\zeta} = g\pa{\eta}+ g\pa{\eta'}$.
In other words, $\zeta$ has infinite genus if and only if one of $\eta$ or $\eta'$ does.
\end{proof}

\subsection{Summary} \label{subsec:summary}

Let $M$ be a \PL\ surface with compact boundary.
Let $N$ be a result of adding a 1-handle at infinity to $M$ along distinct ends $\eta$ and $\eta'$.
The surface $M$ has connected components $M_i$.
If a component $M_i$ does not contain $\eta$ or $\eta'$, then $M_i$ is unchanged by the addition of the 1-handle.
So, it suffices to consider only the component of $N$ which contains the 1-handle.
From now on, assume $N$ is connected.

By Lemma~\ref{lem:partEndSumEnds}, there is a canonical homeomorphism $\eh: \E{N} \cong \E{M} / (\eta \equiv \eta')$.
For all ordinary ends $\alpha$ of $N$, $\alpha$ and $\eh(\alpha)$ are isomorphic as ends, so they have the same invariants.
The single extraordinary end $\zeta$ of $N$ is orientable if and only if both $\eta$ and $\eta'$ are orientable; $\zeta$ has zero genus if and only if both $\eta$ and $\eta'$ have zero genus.

If $\eta$ and $\eta'$ lie in distinct components of $M$, then $N$ is orientable if and only if that $M$ is orientable. Otherwise, $N$ is orientable if and only if $M$ is orientable and the 1-handle is oriented.

If $\eta$ and $\eta'$ lie in distinct components of $M$, then $g(N) = g(M)$.
Otherwise, $g(N) = g(M) + 1$. In either case, $\parity{N} = \parity{M}$.

We conclude: if $M$ is non-orientable, then $N$ is unique up to \PL\ isomorphism; if $M$ is orientable and the 1-handle is oriented, then $N$ is unique up to \PL\ isomorphism; and if $M$ is non-orientable and the 1-handle is not oriented, then $N$ is unique up to \PL\ isomorphism.
That completes the proof of Theorem~\ref{thm:endSumUniquenessPLCase} in \PL.
The \TOP\ and \DIFF\ versions of Theorem~\ref{thm:endSumUniquenessPLCase} are proved below in Section~\ref{sec:secondTheorem}.

\section{Main Theorem by the Uniqueness of Rays} \label{sec:rayUniquenessMain}

In this section, we prove a ray uniqueness result for surfaces. Namely, that all rays within a surface with compact boundary that point to a given end are related by a global isomorphism of the surface. Using that result, we give an alternate and relatively straightforward proof of our main theorem. The authors thank Ric Ancel for suggesting this strategy.

\subsection{Classification of Surfaces with Noncompact Boundary}\label{csncb}

Our first objective is to study and classify all rays $r$ in a given surface with compact boundary up to global isomorphism. We will do that indirectly by removing the interior of a closed regular neighborhood of $r$ to produce a surface with a single noncompact boundary component, and then classify the resulting surface.
We use Brown and Messer's classification of surfaces with possibly noncompact boundary~\cite{brownmesser}.
We recount the invariants in their classification using our notation.

Let $M$ be a connected surface. As in the case of surfaces with compact boundary, the following ``global invariants'' are essential: orientability, genus, parity, and compact boundary components. These invariants are defined exactly as in the compact boundary case, although now the number of compact boundary components may be infinite.

The rest of the invariants required to classify $M$ deal with the ends of $M$ and how those ends interact with its boundary components. As in the case of compact boundary, and end may be zero genus or infinite genus, and it may be orientable or non-orientable.

In surfaces with infinitely many compact boundary components, boundary components may accumulate in certain ends. Define an end of $M$ to be \defword{without compact boundary} provided that a neighborhood of that end contains no compact boundary components of $M$. It can be shown that $M$ has finitely many compact boundary components provided that every end of $M$ is without compact boundary. The surfaces we are interested in have finitely many boundary components, and thus have no ends with compact boundary.

Let $\bdnc{M}$ denote the union of all noncompact boundary components of $M$. The inclusion map $\bdnc{M} \rightarrow M$ induces a map $v: \E{\bdnc{M}} \rightarrow \E{M}$. Each connected component $B$ of $\bdnc{M}$ is a copy of $\R$ and has two ends, say $\alpha$ and $\beta$. The map $e: \E{\bdnc{M}} \rightarrow \pi_0(\bdnc{M})$ is defined by sending $\alpha$ and $\beta$ to $B$. 

We adopt the usual outward normal first convention for orienting the boundary of an oriented manifold.
For example, Figure~\ref{fig:positiveend} shows $M = \R^2_+$ equipped with the standard orientation
corresponding to the ordered basis $(e_1,e_2)$.
\begin{figure}
    \centerline{\includegraphics[scale=0.8]{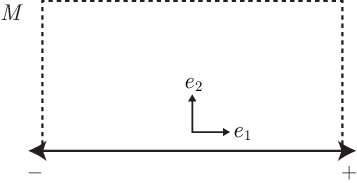}}
    \caption{Surface $M=\mathbb{R}^2_+$ with one noncompact boundary component having a positive end $+$ and a negative end $-$.}
	\label{fig:positiveend}
\end{figure}
The \defword{positive end} of the noncompact boundary component of $M$ is the end labeled $+$.
The \defword{negative end} of the noncompact boundary component of $M$ is the end labeled $-$.
The same conventions and terminology apply to ends of noncompact boundary components of general surfaces in which some disjoint neighborhoods of ends are oriented.

Following Brown and Messer~\cite[pp.~379--381]{brownmesser}, we 
define an \defword{orientation of $\E{\bdnc{M}}$} to be a subset $\mathcal{P} \subset \E{\bdnc{M}}$ as follows.
Given a compact subsurface $C\subset M$, a \defword{complementary domain of $C$} is the closure of one of the components of $M-C$.
\begin{itemize}
	\item If $M$ is orientable, then fix an orientation on $M$. Then, $\mathcal{P}$ consists of the resulting positive ends of $\bdnc{M}$.
	\item If $M$ is non-orientable, then there exists a (nonunique) sequence $Q_1,Q_2,\ldots$ of disjoint orientable complementary domains of compact subsurfaces of $M$ such that every orientable end of $M$ has some (unique) $Q_i$ as a neighborhood. Choose an orientation on each $Q_i$. Then, $\mathcal{P}$ consists of the resulting positive ends of $\bdnc{M}$.
\end{itemize}

It is important to note that there may be many valid orientations of $\E{\bdnc{M}}$ for a given surface $M$.
All of that end-related data can be combined to form a \defword{diagram} as in~\eqref{eq:diagramofsurface}.
Here, $\mathcal{K}$ denotes the set of non-orientable ends, $\mathcal{H}$ denotes the set of non-planar ends,  and $\mathcal{S}$ denotes the set of ends with compact boundary.
\begin{equation} \label{eq:diagramofsurface}
	\begin{tikzcd}
	\pi_0(\bdnc{M}) & \E{\bdnc{M}} \arrow{l}[swap]{e} \arrow{r}{v} & \E{M} & \mathcal{H} \arrow{l} & \mathcal{K} \arrow{l}\\
	& \mathcal{P} \arrow{u} & \mathcal{S} \arrow{u}
	\end{tikzcd}
\end{equation}
Two diagrams $\Delta$ of $M$ and $\Delta'$ of $M'$ are \defword{isomorphic} if there are bijective maps from each object in $\Delta$ to the corresponding object in $\Delta'$ which commute with the arrows of each diagram. The map from $\E{M}$ to $\E{M'}$ must also be a topological homeomorphism. It is important to note that, as  $\mathcal{P}$ is not uniquely defined, it is possible for one surface to have multiple non-isomorphic diagrams. Written in full, the classification theorem~\cite[p.~388]{brownmesser} is as follows.

\begin{theorem}[Classification of Surfaces]\label{thm:classificationOfAllSurfaces}
Let $M$ and $M'$ be two connected surfaces. If $M$ and $M'$ have the same genus, parity, orientability, and number of compact boundary components, and there exist diagrams $\Delta$ of $M$ and $\Delta'$ of $M'$ such that $\Delta \cong \Delta'$, then $M$ and $M'$ are isomorphic.
\end{theorem}

We will need the following strengthening of Theorem~\ref{thm:classificationOfAllSurfaces}.

\begin{theorem}[Addendum to the Classification of Surfaces]\label{thm:classificationOfAllSurfacesAddendum}
Suppose furthermore that we are given an isomorphism $h: \Delta \to \Delta'$.
Then, there is an isomorphism $\psi: M \to M'$ which induces the given isomorphism $h$ of diagrams.
If $M$ and $M'$ are oriented surfaces, and their orientations agree with the orientations of $\Delta$ and $\Delta'$ respectively, then $\psi$ may be chosen to be orientation-preserving.
\end{theorem}

The Addendum follows by Brown and Messer's proof of Theorem~\ref{thm:classificationOfAllSurfaces}.
Their proof~\cite[p.~388--389]{brownmesser} is iterative and begins with the empty function $f_0$.
Given a homeomorphism $f_k:C_k\to C_k'$ (with certain properties) of compact subsurfaces of $M$ and $M'$ respectively,
they extend $f_k$ to a homeomorphism $f_{k+1}:C_{k+1}\to C_{k+1}'$ of larger compact subsurfaces.
If $M$ and $M'$ are oriented, then the first nonempty function $f_1$ may be chosen to respect orientation.
In general, their proof imposes compatibility conditions between the constructed homeomorphisms and isomorphisms of diagrams
(see~\cite[pp.~383--388]{brownmesser}, especially Lemma~2.1, the paragraph on p.~384 before the proof of Lemma~2.1, and the proof of Lemma~2.1).
The present authors found it instructive to run their proof on various examples to gain familiarity with their orientation and compatibility conventions.

In the remainder of this subsection, we present a useful lemma and some applications of the Addendum to demonstrate its utility.
In the next subsection, we prove a ray unknotting theorem (Theorem~\ref{thm:rayUniqueness})
for rays in certain surfaces.
Recall from Section~\ref{sec:endsIntroduction} that a precise definition was given
for a proper ray $r$ to \defword{point to} an end $\eta$ of $M$.
We now expand this definition. Say an end $\tau$ of a (noncompact) boundary component $B$ of $M$ \defword{points to} and end $\eta$ of $M$ if $\E{i}(\tau) = \eta$, where where $i: B \rightarrow M$ is the inclusion map. Equivalently, $\tau$ points to $\eta$ if every ray in $B$ that points to $\tau$ also points to $\eta$.

\begin{lemma}\label{evenendslemma}
Let $M$ be a connected surface and let $\eta$ be an end of $M$.
If $M$ has finitely many noncompact boundary components,
then the number of ends of noncompact boundary components of $M$
that point to $\eta$ is even.
In other words, if $\card{\pi_0(\bdnc{M})}$ is finite, then
$\card{v^{-1}(\eta)}$ is even.
In particular, if $M$ has exactly one noncompact boundary component $B$,
then both ends of $B$ point to the same end of $M$.
\end{lemma}


\begin{proof}
Let $\eta, \tau_1, \dots, \tau_k$ be all the ends of $M$ pointed to by ends of $\bdnc{M}$.
Choose a large compact subsurface $K \subset M$ such that $\closr{M \setminus K}$ is also a subsurface of $M$, and $\eta(K) \neq \tau_i(K)$ for any $i$. Consider $\closr{\eta(K)}$. This is a subsurface of $M$, and the ends of noncompact boundary components of $\closr{\eta(K)}$ pointing to $\eta$ are in one-to-one correspondence with ends of noncompact boundary components of $M$ pointing to $\eta$. Since every end of a noncompact boundary component of $\closr{\eta(K)}$ points to $\eta$ and there are an even number of ends of noncompact boundary components of $\closr{\eta(K)}$, then an even number of ends of noncompact boundary components of $M$ point to $\eta$.
That proves the first conclusion. The second conclusion follows immediately from the first.
\end{proof}

\begin{remark}
The first conclusion of Lemma~\ref{evenendslemma} also holds provided $\eta$ is an isolated end of $M$
and at most finitely many noncompact boundary components of $M$ point to $\eta$.
The proof is similar.
The conclusions of Lemma~\ref{evenendslemma} are false without some restrictions on $M$ or $\eta$.
Consider the surface $M$ depicted in Figure~\ref{fig:oddex} that is obtained from the closed disk by removing a sequence of boundary points and the single limit point of that sequence.
\begin{figure}[htb!]
    \centerline{\includegraphics[scale=0.85]{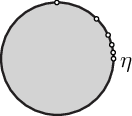}}
    \caption{Surface $M$ obtained from the closed disk by removing a sequence of boundary points and the single limit point of that sequence.}
	\label{fig:oddex}
\end{figure}
The sequence converges to the limit point from one side. The single nonisolated end of $M$ is denoted $\eta$.
That end $\eta$ of $M$ is pointed to by exactly one end of a noncompact boundary component of $M$.
\end{remark}

\begin{corollary}\label{swapendsofb}
Let $M$ be a connected surface with exactly one noncompact boundary component $B$.
Then, there exists an automorphism $\psi:M\to M$ that induces the identity 
on $\E{M}$ and interchanges the two ends of $B$.
\end{corollary}

\begin{proof}
By the previous Lemma~\ref{evenendslemma}, both ends of $B$ point to the same end $\eta$ of $M$.
First, consider the case where $M$ is orientable.
Fix an orientation on $M$ and consider the corresponding diagram $\Delta$ of $M$.
Let $\Delta'$ be the diagram for $M$ with the opposite orientation.
We have an isomorphism of diagrams $h:\Delta\to\Delta'$ that swaps the ends of $B$ and is otherwise the identity.
The desired conclusion now follows by the Addendum (Theorem~\ref{thm:classificationOfAllSurfacesAddendum}).
Second, consider the case where $M$ is not orientable, but $\eta$ is orientable. Then the proof from the first case also applies.
Third, consider the case where $\eta$ is not orientable. Note that $\Delta$ has empty orientation. We have an isomorphism $h: \Delta \rightarrow \Delta$ that swaps the ends of $B$, and the conclusion follows by the Addendum.
\end{proof}

\begin{remark}
If $M$ has more than one noncompact boundary component, then it may not be possible
to interchange the ends of a given noncompact boundary component $B$ of $M$ by an automorphism of $M$.
Consider the strip $[0,1]\times\R$ and connect sum a sequence of projective planes or tori off to one end.
A planar example is in Figure~\ref{fig:oddex} above (see also Dickmann~\cite[Ex.~3.5]{dickmann}).
\end{remark}

A curious question arises: which oriented surfaces admit an orientation reversing automorphism?
The classification of compact, connected, oriented surfaces
(see Theorem~\ref{thm:finite2Manifolds} in Section~\ref{sec:classification})
implies that each such surface admits an orientation reversing automorphism.
(The question has also been studied for closed manifolds in higher dimensions by M\"{u}llner~\cite{mullner}.)
For noncompact surfaces, the situation is more complicated.
We give a positive result as well as two surfaces not admitting an orientation reversing automorphism.

\begin{corollary}
Let $M$ be a connected, oriented, noncompact surface with zero or one noncompact boundary component(s).
Then $M$ admits an orientation reversing automorphism.
\end{corollary}

\begin{proof}
First, consider the case where $M$ has zero noncompact boundary components.
Let $M'$ be $M$ with the opposite orientation.
By the Addendum (Theorem~\ref{thm:classificationOfAllSurfacesAddendum}), there is an orientation-preserving isomorphism from $M$ to $M'$.
That is an orientation-reversing automorphism of $M$. 
Second, consider the case where $M$ has one noncompact boundary component $B$. By the previous corollary, there is an automorphism $\psi: M \rightarrow M$ that interchanges the ends of $B$. 
Evidently, $\psi$ is locally orientation reversing near $B$, and so $\psi$ is globally orientation reversing.
\end{proof}

\begin{remark}
Figure~\ref{fig:noorauto} depicts two oriented surfaces that do not admit an orientation reversing automorphism.
\begin{figure}[htb!]
    \centerline{\includegraphics[scale=0.9]{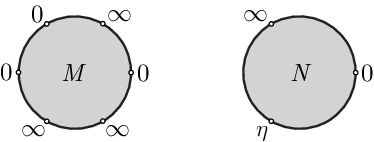}}
    \caption{Surfaces $M$ and $N$ that do not admit orientation reversing automorphisms.}
	\label{fig:noorauto}
\end{figure}
The surface $M$ is obtained from the closed disk (with its standard orientation) by removing six boundary points
and connect summing a sequence of tori at each end marked $\infty$.
So, $M$ has six ends: three of genus zero and three of infinite genus.
The cyclic order of the genera of the ends of $M$ prevent the existence of an orientation reversing automorphism of $M$.
The surface $N$ is obtained from the disjoint union of the closed disk and closed upper half space as follows.
Remove three boundary points from the closed disk,
connect sum a sequence of tori at the end marked $\infty$,
then glue in a sequence of tubes between the end marked $\eta$ and the end of the copy of closed upper half space.
The resulting surface has three ends: one of genus zero, one of infinite genus, and $\eta$. The first two ends each have two ends of noncompact boundary components pointing to them, while $\eta$ is pointed to by four ends of noncompact boundary components. Those properties prevent the existence of an orientation reversing automorphism of $N$. It seems interesting to ask whether there is a reasonable classification of connected, oriented surfaces that admit an orientation reversing automorphism.
\end{remark}

\subsection{Uniqueness of Rays in a Surface}

Although we will use a more technical result to prove the main theorem, it is worth reporting Theorem~\ref{thm:rayUniqueness}, which gives a clean result on the uniqueness of rays in a surface. Many of the ideas in this proof will be used in Section~\ref{subsec:maintheorem2}.

\begin{lemma}\label{lem:surfaceMinusRayEndsDontChange}
Let $M$ be a surface with compact boundary, and $A$ be a closed subset of $M$ isomorphic to $\R^2_+$ that is disjoint from $\bd{M}$. Let $N = \closr{M \setminus A}$. Then, the inclusion map $i: N \rightarrow M$ induces a homeomorphism on the space of ends.
\end{lemma}

\begin{proof}
Identify $A$ with $\R^2_+$. Let $F_j$ be the closed half-disk of radius $j$ centered at $(0,0)$. Note that $(F_j)$ is a compact exhaustion of $A$. Let $(G_j)$ be a compact exhaustion of $N$ by subsurfaces such that $F_j \cap \bd{A} = G_j \cap \bd{A}$, and let $H_j = F_j \cup G_j$.

To prove that the end map is bijective, it suffices to prove that for $j$ large enough, each unbounded connected component of $M \setminus H_j$ contains exactly one unbounded connected component of $N \setminus G_j$. Since end spaces are compact and Hausdorff, any bijective map between end spaces is automatically a homeomorphism.

Let $U$ be any unbounded connected component of $M \setminus H_j$. If $U$ is disjoint from $A$, then $U$ is already an unbounded connected component of $N \setminus G_j$. Assume then that $U$ intersects $A$. Then, $U$ contains $A \setminus F_j$. Since $A \setminus F_j$ borders $N \setminus G_j$, $U$ contains at least one connected component of $N \setminus G_j$. To prove that there is exactly one, it suffices to prove that the two components of $\bd{A} \setminus F_j$ are connected via a path in $N \setminus G_j$. Consider the boundary of the closed subset $\closr{N \setminus G_j}$. This boundary has only two ends, corresponding to the two ends of $\bd{A} \setminus F_j$. Thus, those ends lie in the same connected component (which is a copy of $\R$). Thus, $\bd{A} \setminus F_j$ is connected via paths in $\closr{N \setminus G_j}$. Thus, it is connected via paths in $N \setminus G_j$.
\end{proof}

\begin{theorem}[Ray Uniqueness for Surfaces]\label{thm:rayUniqueness}
Let $M$ be a surface with compact boundary, and let $\eta$ be an end of $M$.
If $r$ and $r'$ are rays in $M$ pointing to $\eta$,
then there is an automorphism $\psi: M \rightarrow M$ which sends $r$ to $r'$.
\end{theorem}

\begin{proof}
It suffices to consider the case where $M$ is connected.
Begin by removing regular neighborhoods $A$ of $r$ and $A'$ of $r'$. Let $N = \closr{M \setminus A}$ and let $N' = \closr{M \setminus A'}$. We first exhibit an isomorphism from $N$ to $N'$ using the classification of surfaces. Then, we will find a compatible isomorphism from $A$ to $A'$. Lastly, we will combine these to define $\psi: M \rightarrow M'$.
It is clear that $M$, $N$, and $N'$ have the same number of compact boundary components.
If $M$ is orientable, then this orientation defines an orientation on $N$ as well. Suppose that $N$ is oriented. There there is an induced orientation on $\bd{A} \subset \bd{N}$. Choose an orientation on $A$ that induces the same orientation on $\bd{A}$. Gluing these orientations together yields an orientation on $M$. Similarly, $N'$ is orientable if and only if $M$ is.

The genus and parity are both defined in terms of a compact exhaustion by subsurfaces. As in Lemma~\ref{lem:surfaceMinusRayEndsDontChange}, identify $A$ with $\R^2_+$. Let $F_j$ be the closed half-disk of radius $j$ centered at $(0,0)$, let $(G_j)$ be a compact exhaustion of $N$ by subsurfaces such that $F_j \cap \bd{A} = G_j \cap \bd{A}$, and let $H_j = F_j \cup G_j$. So, $(H_j)$ is a compact exhaustion of $M$ by subsurfaces.
\[ g(H_j) = \piz{H_j} - \frac{b(H_j) + \chi(H_j)}{2} = \piz{G_j} - \frac{b(G_j) + \chi(G_j)}{2} = g(G_j) \]
As a result, the genus and parity of $N$ equal those of $M$. Similarly, the genus and parity of $N'$ equal those of $M$.

By Lemma~\ref{lem:surfaceMinusRayEndsDontChange}, the inclusion maps $i: N \rightarrow M$ and $i': N' \rightarrow M$ induce homeomorphisms on the end spaces. By looking within a neighborhood disjoint from $A$ and $A'$, it can be shown that for every end of $M$ other than $\eta$, the corresponding ends in $N$ and $N'$ have the same end characteristics.

Given any end $\tau \neq \eta$, the corresponding ends in $N$ and $N'$ are isomorphic. This can be shown by looking in a neighborhood of $\tau$ disjoint from $A$ and $A'$.

It was shown in Section~\ref{subsec:endSumEndOrientation} that removing a strip does not change the orientation of $\eta$. Those arguments apply to this situation as well. It was shown in Section~\ref{subsubsec:endSumEndsGenus} that removing a strip does not change the genus of $\eta$. Those arguments apply to this situation as well.

Recall from Section~\ref{csncb} above that an orientation $\mathcal{P}$ of $\E{\bdnc{N}}$ is a certain subset of $\E{\bdnc{N}}$.
The surface $N$ has a single noncompact boundary component $B$, and both ends of $B$ point to the same end $\beta$ of $N$ by Lemma~\ref{evenendslemma}.
If $\beta$ is nonorientable, then $\mathcal{P}=\emptyset$.
If $\beta$ is orientable, then choose an orientation on an orientable complementary domain that contains $\beta$; this yields $\mathcal{P}=\{\tau\}$ where $\tau$ is one of the two ends of $B$.
Similarly, an orientation $\mathcal{P}'$ of $\E{\bdnc{N'}}$ is defined.
Note that $\eta$ is orientable if and only if $\beta$ is orientable, and
$\eta$ is orientable if and only if $\beta'$.
Thus, $\mathcal{P}$ and $\mathcal{P}'$ are either both empty or are both singletons.
In either case, there is a unique bijection $\mathcal{P} \rightarrow \mathcal{P}'$.
That bijection extends to an isomorphism $\E{\bdnc{N}} \rightarrow \E{\bdnc{N'}}$.
Altogether, this data can be used to construct an isomorphism of diagrams for $N$ and $N'$.

By the Classification of Surfaces (Theorem~\ref{thm:classificationOfAllSurfaces}), there exists an isomorphism $\phi: N \rightarrow N'$. This isomorphism must send $\bd{A}$ to $\bd{A'}$. To extend $\phi$ to an automorphism of $M$, it suffices to find a compatible isomorphism from $A$ to $A'$. But any isomorphism $\bd{\R^2_+} \rightarrow \bd{\R^2_+}$ extends to an isomorphism $\R^2_+ \rightarrow \R^2_+$.
The resulting automorphism of $M$ sends $r$ to $r'$.
\end{proof}

\begin{remark}\label{rem:rayUniquenessFailure}
The conclusion of Theorem \ref{thm:rayUniqueness} does not hold in general if $M$ has noncompact boundary components. Consider the surface $M$ which is a closed disk with one boundary point removed and a sequence of 1-handles attached as in Figure~\ref{fig:raynotunique1}.
	\begin{figure}[htb!]
		\centerline{\includegraphics[scale=0.85]{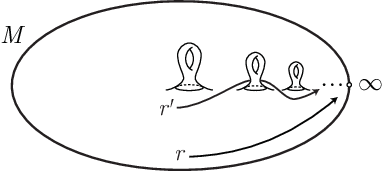}}
		\caption{Surface $M$ containing rays $r$ and $r'$.}
		\label{fig:raynotunique1}
	\end{figure}
(Dickmann~\cite{dickmann} called $M$ a \textit{sliced Loch Ness monster} and studied the mapping class groups of $M$ and similar surfaces.)
So, $M$ has one end $\eta$ and that end has infinite genus.
Let $r$ be a ray parallel to $\bd{M}$, and let $r'$ be a ray that winds around the glued in 1-handles as in Figure~\ref{fig:raynotunique1}.
Suppose there is an automorphism $\psi$ of $M$ that sends $r$ to $r'$.
Let $A$ and $A'$ be closed regular neighborhoods of $r$ and $r'$, respectively.
By the uniqueness of regular neighborhoods, $\psi$ can be modified so that it also sends $A$ to $A'$.
Thus, $N = \closr{M \setminus A}$ is isomorphic to $N' = \closr{M \setminus A'}$.
However, $N$ has one zero-genus end and one infinite-genus end, while $N'$ has two infinite-genus ends as in Figure~\ref{fig:raynotunique2}.
	\begin{figure}[htb!]
		\centerline{\includegraphics[scale=0.9]{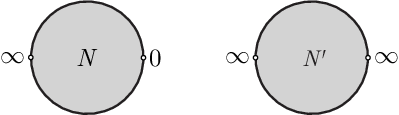}}
		\caption{Nonisomorphic surfaces $N$ and $N'$ each obtained from $M$ by removing the interior of a regular neighborhood of a ray.}
		\label{fig:raynotunique2}
	\end{figure}
This is a contradiction, so $r$ and $r'$ are not related by an automorphism of $M$.
\end{remark}

\subsection{Main Theorem, Re-Proved} \label{subsec:maintheorem2}

Using the classification of all surfaces, with our addendum, we are ready to dive into the main theorem once again.

\begin{lemma} \label{lem:regNeighborhoodUniqueness}
	Let $M$ be a surface with compact boundary. Let $\eta$ and $\eta'$ be distinct ends of $M$. For $i = 1, 2$, let $A_i, A_i'$ be disjoint closed subsurfaces of $M$ isomorphic to $\R^2_+$ and disjoint from $\bd{M}$. Suppose as well that that $A_i$ points to $\eta$ and $A_i'$ points to $\eta'$. Let $\mu: A_1 \rightarrow A_2$ and $\mu': A_1' \rightarrow A_2'$ be isomorphisms. If any of the following conditions hold, then there is an automorphism $\psi$ of $M$ such that the restriction of $\psi$ to $A_1$ is $\mu$ and the restriction of $\psi$ to $A_1'$ is $\mu'$:
\begin{enumerate*}[label=(\roman*)]
\item $M$ is connected and orientable, and $\mu$, $\mu'$ are both orientation-preserving,
\item $M$ is connected and non-orientable, or
\item $\eta$ and $\eta'$ lie in distinct connected components of $M$.
\end{enumerate*}
\end{lemma}

\begin{proof}
Let $N_i = \closr{M \setminus (A_i \cup A_i')}$. We will show that $N_1$ and $N_2$ are isomorphic, and that this isomorphism is compatible with $\mu$ and $\mu'$.

Using the same strategies as Theorem~\ref{thm:rayUniqueness}, it can be shown that $N_1$ and $N_2$ have the same genus, parity, orientability, and number of compact boundary components as $M$.

Using the same strategies as Theorem~\ref{thm:rayUniqueness}, it can be shown that the inclusion maps $i_i: N_i \rightarrow M$ induce homeomorphisms on the space of ends, and that corresponding ends have the same genus, orientability, and number of compact boundary components.

By restricting $\mu$ and $\mu'$ to $\bd{A_1}$ and $\bd{A_1'}$, we obtain a map from $\bd{A_1} \cup \bd{A_1'} \rightarrow \bd{A_2} \cup \bd{A_2'}$. Define $h: \E{\bdnc{N_1}} \rightarrow \E{\bdnc{N_2}}$ to be the corresponding map of ends.

Suppose first that $M$ is connected and orientable. Fix an orientation on $M$, which induces orientations on $N_1$ and $N_2$. Let $\mathcal{P}_i$ be the orientation of $\E{\bdnc{N}}$ induced by the orientation on $N_i$. Since $\mu$ and $\mu'$ are orientation-preserving, the map $h$ sends $\mathcal{P}_1$ to $\mathcal{P}_2$. Using $h$, we can construct an isomorphism of diagrams. By the Addendum (Theorem~\ref{thm:classificationOfAllSurfacesAddendum}), there is an orientation-preserving isomorphism $\phi: N_1 \rightarrow N_2$ that induces this isomorphism of diagrams.

Since $\mu$ and $\mu'$ are both orientation-preserving, $\phi$ has the same effect on the orientations of $\bd{A_1}$ and $\bd{A_1'}$ as $\mu$ and $\mu'$ do. Because of this, $\phi$ can be isotoped in a neighborhood of $\bd{A_1}$ and $\bd{A_1'}$ to become equal to $\mu$ and $\mu'$ on $\bd{A_1}$ and $\bd{A_1'}$.

Now suppose that $M$ is connected and non-orientable. If $\eta$ is orientable, then add an end of $\bd{A_1}$ to $\mathcal{P}_1$, and add the corresponding end of $\bd{A_2}$ (under $\mu$) to $\mathcal{P}_1$. Do the same for $\eta'$. Now, $h$ sends $\mathcal{P}_1$ to $\mathcal{P}_2$, so we can again obtain an isomorphism $\phi: N_1 \rightarrow N_2$ that is compatible with $h$. And, $\phi$ can be isotoped in a neighborhood of $\bd{A_1}$ and $\bd{A_1'}$ to become equal to $\mu$ and $\mu'$ on $\bd{A_1}$ and $\bd{A_1'}$.

Now suppose that $\eta$ and $\eta'$ lie in distinct connected components of $M$. As before, construct $\mathcal{P}_1$ and $\mathcal{P}_2$, and obtain an isomorphism $\phi: N_1 \rightarrow N_2$ that is compatible with $h$. Isotope $\phi$ in a neighborhood of $A_1$ and $A_1'$ to become equal to $\mu$ and $\mu'$ on $\bd{A_1}$ and $\bd{A_1'}$.

Once $\phi$ agrees with with $\mu$ and $\mu'$ on $A_1$ and $A_1'$, all three maps can be glued together to form an isomorphism $\psi: M \rightarrow M$.
\end{proof}

As a corollary, we recover the main theorem.

\begin{theorem}
Let $M$ be a PL surface with compact boundary, and let $\eta$, $\eta'$ be distinct ends of $M$. Let $N$ be a result of adding a 1-handle at infinity to $M$ along $\eta$ and $\eta'$. If $\eta$ and $\eta'$ are ends of distinct connected components of $M$, then $N$ is unique up to isomorphism. If $\eta$ and $\eta'$ are ends of the same connected component $M_1$ of $M$, and $M_1$ is non-orientable, then $N$ is unique up to isomorphism. Lastly, suppose that $\eta$ and $\eta'$ are ends of the same connected component $M_1$ of $M$, and $M_1$ is orientable. If the 1-handle is oriented, then $N$ is unique up to isomorphism. If the 1-handle is not oriented, then $N$ is unique up to isomorphism.
\end{theorem}

\begin{proof}
Consider two end sums, $N_1$ and $N_2$, along $\eta$ and $\eta'$, satisfying the above criteria. Let $\nu_i, \nu_i'$ be the tubular neighborhood maps used to construct $N_i$, and let $A_i, A_i'$ be images of $\nu_i, \nu_i'$ respectively. Let $\psi: A_1 \rightarrow A_2$ be the canonical isomorphism between regular neighborhoods, defined by $\psi = \nu_2 \circ \nu_1^{-1}$. Similarly, define $\psi': A_1' \rightarrow A_2'$.
	
Using Lemma~\ref{lem:regNeighborhoodUniqueness}, we extend $\psi$ and $\psi'$ to a global automorphism $\chi: M \rightarrow M$. Since $\chi$ preserves all data used in the 1-handle construction, $\chi$ extends to an isomorphism from $N_1$ to $N_2$.
\end{proof}

\section{Extension of the Main Theorem to TOP and DIFF}
\label{sec:secondTheorem}

Our main theorem---Theorem~\ref{thm:endSumUniquenessPLCase}---also holds true in the \TOP\ and \DIFF\ categories. Hence, Corollary~\ref{cor:endSumPL} does as well. This is an aspect of a general theme that often statements about surfaces are true independent of category. An important rationale for this theme, and a key reason why we can generalize in this case, is that every surface has a unique \TOP, \PL, and \DIFF\ structure up to isomorphism. See, for example, Moise~\cite[Preface \& Ch.~8]{moise}, Thurston~\cite[$\S$3.10]{thurston}, and Hatcher~\cite{hatcher}. In this section, we show that Theorem~\ref{thm:endSumUniquenessPLCase} for \PL\ implies the corresponding \TOP\ and \DIFF\ analogues.

First, suppose we are given \TOP\ data for the addition of a $1$-handle at infinity as in Theorem~\ref{thm:endSumUniquenessPLCase}.
We reuse the notation from Conventions~\ref{keyconventions} in Section~\ref{sec:mainTheorem}.
In particular, $N$ is the \TOP\ surface that results from adding a \TOP\ $1$-handle at infinity to the surface $M$ along ends $\eta$ and $\eta'$.
The $1$-handle is added using disjoint rays $r$ and $r'$, pointing to the ends $\eta$ and $\eta'$ respectively, and disjoint tubular neighborhood maps $\nu$ and $\nu'$.
By the triangulation of surfaces, there exists a \PL\ surface $M_{\PL}$ and a homeomorphism $j:M\to M_{\PL}$. So, $j\circ r$ and $j\circ r'$ are disjoint \TOP\ rays in $M_{\PL}$, and $j\circ \nu$ and $j\circ \nu'$ are disjoint \TOP\ tubular neighborhood maps.
We may adjust $M_{\PL}$ by a small ambient homeomorphism with with support in any prescribed open neighborhood of $\image{j\circ \nu} \cup \image{j\circ \nu'}$ such that the images of the tubular neighborhood maps $j\circ \nu$ and $j\circ \nu'$ are \PL. That is, by Moise~\cite[Thm.~10.13]{moise}, there is a homeomorphism $k:M_{\PL}\to M_{\PL}$ with support in any prescribed open neighborhood of $\image{j\circ \nu} \cup \image{j\circ \nu'}$ such that the homeomorphism $h=k\circ j$ satisfies: (i) the disjoint rays $h\circ r$ and $h\circ r'$ are \PL\ embeddings and (ii) the disjoint tubular neighborhood maps $h\circ\nu$ and $h\circ\nu'$ are \PL\ embeddings.
As $j$ and $k$ are homeomorphisms, each induces a homeomorphism on the spaces of ends; note that $k$ induces the identity on the space of ends of $M_{\PL}$. Let $\eta_{\PL}$ and $\eta'_{\PL}$ denote the ends of $M_{\PL}$ corresponding under $j$ to $\eta$ and $\eta'$ respectively.

We now have \PL\ data for Theorem~\ref{thm:endSumUniquenessPLCase} consisting of the surface $M_{\PL}$, disjoint rays $h\circ r$ and $h\circ r'$, pointing to the ends $\eta_{\PL}$ and $\eta'_{\PL}$ respectively, and disjoint tubular neighborhood maps $h\circ\nu$ and $h\circ\nu'$. Let $N_{\PL}$ be the \PL\ surface that results from adding a \PL\ $1$-handle at infinity to $M_{\PL}$ according to this data. Note that we have an induced homeomorphism $\alpha:N\to N_{\PL}$.

Now, suppose we have possibly different \TOP\ data pointing to the same ends of $M$. That is, we have disjoint rays $s$ and $s'$, pointing to the \textit{same} ends as before $\eta$ and $\eta'$ respectively, and disjoint tubular neighborhood maps $\mu$ and $\mu'$. Let $P$ be the \TOP\ surface that results from adding a \TOP\ $1$-handle at infinity to $M$ according to this data. We must show that $N$ and $P$ are homeomorphic. We have the homeomorphism $j:M\to M_{\PL}$. As above, there is a homeomorphism $l:M_{\PL}\to M_{\PL}$ with support in any prescribed open neighborhood of $\image{j\circ \mu} \cup \image{j\circ \mu'}$ such that the homeomorphism $g=l\circ j$ satisfies: (i) the disjoint rays $g\circ s$ and $g\circ s'$ are \PL\ embeddings and (ii) the disjoint tubular neighborhood maps $g\circ\mu$ and $g\circ\mu'$ are \PL\ embeddings. Let $P_{\PL}$ be the \PL\ surface that results from adding a \PL\ $1$-handle at infinity to $M_{\PL}$ according to this data. Note that we have an induced homeomorphism $\beta:P\to P_{\PL}$.

Observe that in $M_{\PL}$, the rays $h\circ r$ and $g\circ s$ both point to $\eta_{\PL}$, and the rays $h\circ r'$ and $g\circ s'$ both point to $\eta'_{\PL}$. Therefore, the \PL\ version of Theorem~\ref{thm:endSumUniquenessPLCase} yields a \PL\ homeomorphism $\gamma:N_{\PL}\to P_{\PL}$. Hence, $\beta^{-1}\circ\gamma\circ\alpha:N\to P$ is a homeomorphism as desired. This completes our proof of the \TOP\ version of Theorem~\ref{thm:endSumUniquenessPLCase}.

Second, suppose we are given two collections of \DIFF\ data for the addition of a $1$-handle at infinity as in Theorem~\ref{thm:endSumUniquenessPLCase}. One collection is the surface $M$, disjoint rays $r$ and $r'$, pointing to the ends $\eta$ and $\eta'$ respectively, and disjoint tubular neighborhood maps $\nu$ and $\nu'$. Let $N$ be the \DIFF\ surface that results by adding a \DIFF\ $1$-handle at infinity to $M$ according to this data. The other collection is the surface $M$, disjoint rays $s$ and $s'$, pointing to the \textit{same} ends as before $\eta$ and $\eta'$ respectively, and disjoint tubular neighborhood maps $\mu$ and $\mu'$. Let $P$ be the \DIFF\ surface that results by adding a \DIFF\ $1$-handle at infinity to $M$ according to this data. We must show that $N$ and $P$ are diffeomorphic. Ignoring \DIFF\ structures for the moment, the \TOP\ version of Theorem~\ref{thm:endSumUniquenessPLCase}---proved above---implies that $N$ and $P$ are homeomorphic. As $N$ and $P$ are homeomorphic \DIFF\ surfaces, they are diffeomorphic as desired.
That completes our proof of the \DIFF\ version of Theorem~\ref{thm:endSumUniquenessPLCase}.

\bibliographystyle{amsalpha}

\end{document}